\documentclass[11pt]{amsart}
\title{Auslander-Reiten sequences for Gorenstein rings of dimension one 
} 
\author{Robert Roy
}

\usepackage{amssymb, amscd, amsmath, float, graphicx, longtable,color,
  enumerate, fullpage
}
\usepackage[textwidth=2cm,textsize=small]{todonotes}
\usepackage{amsfonts}
\usepackage{hyperref}
\usepackage{verbatim}
\usepackage{mathrsfs}
\usepackage{psfrag}
\usepackage{fouriernc}\DeclareMathAlphabet{\mathcal}{OMS}{cmsy}{m}{n}
\usepackage[all]{xy}
\SelectTips{eu}{} 
\xyoption{curve}
\date{}
\usepackage{amssymb}
\usepackage{fullpage}

\DeclareMathOperator{\Irr}{Irr}
\DeclareMathOperator{\length}{length}
\DeclareMathOperator{\push}{push}
\DeclareMathOperator{\Aut}{Aut}
\DeclareMathOperator{\avg}{avg}
\DeclareMathOperator{\CM}{CM}

\DeclareMathOperator{\End}{End}

\DeclareMathOperator{\Hom}{Hom}
\DeclareMathOperator{\Ext}{Ext}
\DeclareMathOperator{\syz}{syz}

\DeclareMathOperator{\cok}{cok}
\DeclareMathOperator{\im}{im}

\DeclareMathOperator{\soc}{soc}
\DeclareMathOperator{\Ann}{Ann}
\DeclareMathOperator{\rank}{rank}

\DeclareMathOperator{\id}{id}
\DeclareMathOperator{\trace}{trace}

\DeclareMathOperator{\Mat}{Mat}

\renewcommand{\to}{\longrightarrow}

\newcommand{\ra}{\rightarrow}

\newcommand{\NN}{\mathbb{N}}
\newcommand{\ZZ}{\mathbb{Z}}
\newcommand{\QQ}{\mathbb{Q}}

\newcommand{\Rbar}{\overline{R}}
\newcommand{\Rhat}{\hat{R}}
\newcommand{\Dbar}{\overline{D}}
\newcommand{\Dhat}{\hat{D}}

\newcommand{\Bbar}{\overline{B}}
\renewcommand{\hbar}{\overline{h}}

\newcommand{\gbar}{\overline{g}}
\newcommand{\fbar}{\overline{f}}

\newcommand{\stEnd}{\underline{\End}}

\newcommand{\m}{\mathfrak{m}}
\newcommand{\n}{\mathfrak{n}}
\newcommand{\into}{\hookrightarrow}
\newcommand{\onto}{\twoheadrightarrow}
\newcommand{\p}{\mathfrak{p}}
\newcommand{\q}{\mathfrak{q}}
\newcommand{\J}{\mathcal{J}}

\newcommand{\gpt}{\tilde{\gamma'}}
\newcommand{\Rpb}{\overline{R'}}
\newcommand{\F}{\mathfrak{F}}

\theoremstyle{plain}
\newtheorem{theorem}{Theorem}

\newtheorem{prop}[theorem]{Proposition}

\newtheorem{lem}[theorem]{Lemma}

\newtheorem*{theorem*}{Theorem}
\newtheorem*{prop*}{proposition}

\theoremstyle{definition}

\newtheorem{defn}[theorem]{Definition}

\newtheorem{remark}[theorem]{Remark}
\newtheorem{notation}[theorem]{Notation}

\newtheorem{sit}[theorem]{}

\newtheorem{caveat}[theorem]{Caveat}
\numberwithin{theorem}{section}
\numberwithin{equation}{section}

\begin{document}
\large
\maketitle

\bibliographystyle{amsplain}




\begin{abstract}

Let $R$ be a complete  local Gorenstein ring of dimension one, with maximal ideal $\m$. We show that if $M$ is a Cohen-Macaulay $R$-module which begins an AR-sequence, then this sequence is produced by a  particular endomorphism of $\m$ corresponding to a minimal prime ideal of $R$. We apply this result to determining the shape of some components of stable Auslander-Reiten quivers, which in the considered examples are shown to be tubes. 

\end{abstract}

\section{Introduction}
\label{intro}
The theory of Auslander-Reiten (AR) quivers is central in the study of artin algebras. Regarding AR theory for maximal Cohen-Macaulay  modules over a complete Cohen-Macaulay local ring, the cases of finite AR quivers have been studied thoroughly (see \cite{Yoshino:book}), but in the more common case of infinite type, shapes of AR quivers seem to be largely unknown. 

The paper \cite{AKM} agrees with this assessment (cf. its introduction),  and begins to bridge this gap. It establishes a variety of lemmas in the  context of symmetric orders over a DVR, $\mathcal{O}$, and applies  these lemmas to proving the shape (namely, a tube) of some components of the AR quiver of a truncated polynomial ring $\mathcal{O}[X]/(X^n)$. The  work of the latter calculation consisted largely in proving that a module is indecomposable, but also in proving that a sequence is an AR sequence. 
Our main result, Theorem~\ref{main theorem}, makes the latter task much easier, in the setting of a complete local Gorenstein ring $R$ of dimension one, with maximal ideal $\m$. Namely it shows that there exists a set of elements of $ \End_R \m $, corresponding to the minimal primes of $R$, which produce the AR sequences in a concrete way.  We apply the theorem in Section 6, where we establish the shapes of AR components (again, tubes) over a graded hypersurface of the form $k[x,y]/((bx^p+y^q)f)$ where  $f \in k[x,y]$ is an arbitrary homogeneous polynomial.

Regarding the structure of this paper, Sections 2 and 3, as well as the Appendix, consist of supporting material for the proof of our main result  in Section 4.   In Section 5 we give background concerning the abstract notion of stable translation quivers, and in Section 6 we record lemmas concerning the AR quiver of $R$.  Specifically,  Proposition~\ref{tube proposition} establishes a criterion, based on material in  \cite{AKM}, for proving that an AR component is a tube. In Section 7 we apply this proposition and our main result, to an example.

We would like to thank the developers of Singular \cite{DGPS}; we used it to compute many examples testing Theorem~\ref{main theorem} and Proposition~\ref{syz gamma} before  proving them.

\section{Preliminaries} \label{sec:prelims}

\notation \label{initial notation}

Throughout this paper, all rings are assumed noetherian. A ring which is described with any subset of the words \{reduced, Cohen-Macaulay, Gorenstein, regular\} is implicitly commutative.  By a \emph{graded ring} we will mean a $\ZZ$-graded ring, that is, a ring $A=\bigoplus_{i \in \ZZ} A_i$ satisfying $A_iA_j \subseteq A_{i+j}$. If $A$ has not been referred to as a graded ring, $\J(A)$ will  denote the Jacobson radical of $A$, whereas if $A$ is explicitly graded, $\J(A)$ will denote the intersection of all maximal graded left ideals of $A$ (in our situations this will always coincide with the intersection of all maximal graded right ideals). Similarly, but when $A$ is commutative, if $A$ is not given a grading then $Q(A)= A[\text{nonzerodivisors}]^{-1}$ (the total quotient ring of $A$), whereas if $A$ is graded then $Q(A)=A[\text{homogeneous nonzerodivisors}]^{-1}$.

A (graded) ring $A$ is said to be  (graded-) \emph{local} if $A/\J(A)$ is a division ring.  By a \emph{connected graded} ring we shall mean a commutative $\NN$-graded ring $R=\bigoplus_{i \ge 0} R_i$ such that $R_0$ is a field. In this case $\Rhat$ will denote the $\m$-adic completion of $R$, where $\m=\bigoplus_{i \ge 1} R_i$. If $R$ is any commutative ring, $\min R$ will denote its set of minimal primes.  If $R$ is  furthermore  local or graded-local, we usually write $\m_R$ instead of $\J(R)$.
%

  We will say that an $R$-module $M$ \emph{has rank} (specifically, rank $n$), if $M \otimes_R Q(R)$ is a free $Q(R)$-module (of rank $n$). If $R$ is  reduced, then $\overline{R}$ will denote the integral closure of $R$ in $Q(R)$.

\subsection{Trace lemmas.} \label{sec:trace lemmas}

We establish some preliminary lemmas regarding trace. Observations of this general type have certainly been made before; see \cite[proposition 5.4]{Auslander:rationalsing}. 
First, we define the trace of an endomorphism of an arbitrary finitely generated projective module, as in \cite{hattori1965}.
\defn \label{defn:trace}  Let $A$ be a commutative ring, and let $P$ be a finitely generated projective $A$-module. Then the map $\mu_P \colon \Hom_A(P,A) \otimes_A P \to \End_A P$ given by $f \otimes x \mapsto(y \mapsto f(y)x)$ is an isomorphism, by Lemma~\ref{lemma:stable hom}. Let $\epsilon \colon  \Hom_A(P,A) \otimes_A P \to A$ denote the map given by $f \otimes x \mapsto f(x)$. For $h \in \End_A P$, we define $\trace(h)=\epsilon(\mu_P^{-1}(h))$. If $e_1,\dots ,e_n$ and $\varphi_1, \dots, \varphi_n \in \Hom_A(P,A)$ are such that $\mu_P(\sum_{i=1}^n \phi_i \otimes e_i)=\id_P$, then $\trace(h)$ furthermore equals $\sum_{i=1}^n \phi_i(h(e_i))$. From this, and using that $P=\sum_i A e_i$, it follows that trace is symmetric, in the sense that $\trace(gh)=\trace(hg)$ for all $g,h \in \End_A P$.

\begin{remark} \label{rmk:trace} We can see that the above definition of trace specializes to the usual one when $P$ is free, by taking the aforementioned $\{e_i,\phi_i\}_i$ to be a free basis and the corresponding projection maps. If $A=k_1 \times \ldots \times k_s$ is a product of fields $k_i$, then by a similar argument we see that for any $h \in \End_A P$, we have $\trace(h)=(\trace(h \otimes_A k_1), \ldots, \trace(h \otimes_A k_s))$. \end{remark}

Let $R$ be a commutative ring and set $Q=Q(R)$. Recall that if $R$ is an ungraded reduced ring, then $Q$ is the product of fields $R_\p=Q(R/\p)$ where $\p$ ranges over $\min R$. In particular, each $R_{\p}$ is an ideal of $Q$, and a $Q$-algebra.

\begin{lem} \label{trace is integral} Let $R$ be a reduced ring (possibly graded), let $M$  be a finitely generated $R$-module such that $M \otimes_R Q$ is $Q$-projective, and let $h \in \End_{R} M$. Then  $\trace(h \otimes_{R} Q) \in \Rbar$. (In the ungraded case, the condition that $M \otimes_R Q$ is $Q$-projective is automatically satisfied, since $Q$ is semisimple.)\end{lem}

\begin{proof} First suppose the graded case. Let $Q'=R[\text{nonzerodivisors}]^{-1}$; thus $Q'$ is a localization of $Q$ is a localization of $R$, and $R \subseteq Q \subseteq Q'$. As $M \otimes_R Q$ is $Q$-projective, there exists a finite set $\{e_i \in M \otimes_R Q\}_i$ and corresponding $\{\varphi_i \colon M \otimes_R Q \to Q\}_i$ such that $y=\sum_i \phi_i(y)e_i$ for all $y \in M \otimes_R Q$. Then the images of the $e_i$ in $M \otimes_R Q'$ have the property that $y=\sum_i (\phi_i \otimes_Q Q') (y)e_i$ for all $y \in M \otimes_R Q'$. Therefore $\trace(h \otimes_R Q)=\sum_i \phi_i((h\otimes_R Q )(e_i))= \sum_i (\phi_i\otimes_Q Q')((h\otimes_R Q' )(e_i))=\trace(h \otimes_R Q')$. Since $\Rbar$ is equal to the integral closure of $R$ in $Q'$ by \cite[Corollary 2.3.6]{Swanson-Huneke}, we are thus reduced to the ungraded case.

 Since $\Rbar=\prod_{\p \in \min R} \overline{R/\p}$, we see by Remark~\ref{rmk:trace} that it suffices to show  $\trace(h \otimes_{R} R_{\p}) \in \overline{R/\p}$, for each $\p \in \min R$. As $h \otimes_{R} R_{\p}=(h \otimes_{R} R/\p) \otimes_{R/\p} R_{\p}$,  we may assume $R$ is a domain. By  \cite[Theorem 2.1]{Matsumura}, $h$ satisfies a monic polynomial with coefficients in $R$, say $f(X) \in R[X]$. Let $H =h \otimes_{R} Q$, and let $\mu(X) \in Q[X]$ be the minimal polynomial of $H$. Let $\chi(X) \in Q[X]$ be the characteristic polynomial of $H$, and take a field extension $L \supseteq Q$ over which $\chi(X)$ splits, say $\chi(X)=(X-\alpha_1)(X-\alpha_2) \cdots  (X-\alpha_s)$, $\alpha_i \in L$. Each $\alpha_i$ is also a root of $\mu(X)$, and therefore of $f(X)$.  Therefore $R[\alpha_1,...,\alpha_s]$ is an integral extension of $R$. Thus $\Rbar \supseteq Q \cap R[\alpha_1,...,\alpha_s]$, which contains the coefficients of $\chi(X) $. Finally, recall that $\trace(H)$ is the degree-($s-1$) coefficient of $\chi$. \qed \end{proof}

\begin{lem} \label{trace-in-radical} In the situation of Lemma~\ref{trace is integral},  assume that $\dim R=1$ and  that   $R$  is either a complete local ring or a connected graded ring. If some power of $h$ lies in $\m_R \End_{R} M$, then $\trace(h \otimes_{R} Q) \in \J(\Rbar)$. \end{lem}

\begin{proof} As $\J(\Rbar)=\prod_{\p \in \min R} \J(\overline{R/\p})$, we may again assume $R$ is a domain. In the connected graded case, we have $\J(\overline{\Rhat}) \cap \Rbar =\J(\Rbar)$ by Lemma~\ref{bar commutes with hat}, and we can therefore assume the complete local case.  We can also assume $M \subseteq M\otimes_R Q$, i.e., replace $M$ by its image in $M\otimes_R Q$. Now let $M\Rbar$ denote the $\Rbar$-module of $M \otimes_{R} Q$ generated by $M$. Note that $M \Rbar$ is a free $\Rbar$-module, since all torsion-free $\Rbar$-modules are free. Since $\Rbar$ is local, we can choose a basis for $M \Rbar$ which consists of elements of $M$, say $e_1,...,e_n$. (Indeed, setting $n=\rank(M \Rbar)$, Nakayama's lemma allows us to find a  set $\{e_1,...,e_n\} \subset M$ such that $M\Rbar=\sum_i \Rbar e_i$. Then we have a surjective endomorphism of $M\Rbar$, equivalently an automorphism, mapping a basis onto $\{e_1,...,e_n\}$.)  By fixing this basis, we can identify $\End_{R} M$ as an $R$-subalgebra of the ring of $n \times n$ matrices $\Mat_{n \times n} (\Rbar)$, in the obvious way. By assumption on $h$, some power of $h$ lies in $\m_R \Mat_{n \times n} (\Rbar) \subseteq \J(\Rbar) \Mat_{n \times n}(\Rbar)$. Thus the image of $h$ in $\Mat_{n \times n}(\Rbar/\J(\Rbar))$ is nilpotent. The lemma now follows from the fact that over a field, any nilpotent matrix has zero trace. \qed \end{proof}

\subsection{Cohen-Macaulay modules and Gorenstein rings.} \label{subsec:CM and Gor}

For the remainder of this section, assume $R$ is a Cohen-Macaulay ring which is either a complete local ring or a connected graded ring. Let $\m=\m_R$. \begin{notation}
In the complete local case, let $\CM (R)$ denote the category of finitely generated maximal Cohen-Macaulay  $R$-modules, and (following~\cite{Auslander-Reiten:gradedCM}) let $L_p( R)$ denote the full subcategory of $\CM (R)$ whose objects $M$ have the property that $M_\p$ is $R_\p$-free for all   nonmaximal prime ideals $\p$. 
If  $R$ is instead a Cohen-Macaulay connected graded ring, we define $\CM(R)$ and $L_p(R)$ in the same way but restricted to graded modules (keeping the full Hom sets). \end{notation}

Let  $k=R/\m_R$. Then $R$ is  \emph{Gorenstein} if  and only if it is Cohen-Macaulay and $\dim_k (\Ext^{\dim R}_R(k,R))=1$.  If $R$ is Gorenstein, and $M \in \CM(R)$, then (\cite[Theorem 3.3.10]{BH}): $\Ext^i_R(M,R)=0$ for all $i \ge 1$, and the map $M \to \Hom_R( \Hom_R(M,R),R)$ given by $m \mapsto (f \mapsto f(m))$ is an isomorphism.  We  denote $\Hom_R(M,R)$ by $M^*$.

\subsection{Auslander-Reiten sequences.} \label{subsec:AR seq}

\defn \label{defn:AR-seq} Let $N$ be an indecomposable in $\CM(R)$. Then (cf. \cite[Lemma $2.9'$]{Yoshino:book}) we may define an \emph {Auslander-Reiten (AR) sequence starting from $N$} to be a short exact sequence
 \begin{equation}\label {eqn:AR seq} \xymatrix{ 0 \ar[r] & N \ar[r]^p & E \ar[r]^q & M \ar[r] & 0 }\end{equation}  
in $\CM(R)$ such that $M$ is indecomposable and the following property is satisfied: Any map $N \to L$ in $\CM(R)$ which is not a split monomorphism factors through $p$. Equivalently, $N$ is indecomposable and any map $L \to M$ in $\CM(R)$ which is not a split epimorphism factors through $q$. The sequence \eqref{eqn:AR seq} is unique if it exists, and is also called the AR sequence ending in $M$. Given an AR sequence~\eqref{eqn:AR seq}, $N$ is called the Auslander-Reiten translate of $M$, written $\tau(M)$; and $\tau^{-1}(N)$ denotes $M$. 

\defn \label{defn:irreducible} A morphism $f: X \to Y$ in $\CM(R)$ is called an \emph{irreducible morphism}  if (1) $f$ is neither a split monomorphism nor a split epimorphism, and (2) Given any pair of morphisms $g$ and $ h$ in  $\CM(R)$ satisfying $f=gh$, either $g$ is a split epimorphism or $h$ is a split monomorphism.

\sit \label{sit:irreducible} Let $L$ be an indecomposable in $\CM(R)$, and assume we have the AR sequence~\eqref{eqn:AR seq}. Then the following are equivalent (cf.  \cite[2.12 and $2.12'$] {Yoshino:book}):

 (a) $L $ is isomorphic to a direct summand of $E$.
 
 (b) There exists an irreducible morphism $N \to L$.
 
 (c) There exists an irreducible morphism $L \to M$.

\begin{theorem} \label{Yo theorem} (cf. \cite[Theorem 3.4]{Yoshino:book}, and  \cite[Theorem 3]{Auslander-Reiten:gradedCM}) Let $M \ncong R$ be an indecomposable in $ \CM(R)$. Then  $M \in L_p(R)$ if and only if there exists an AR sequence ending in $M$. \end{theorem}

Notice also that if $R$ is Gorenstein, applying $(\_)^*$ shows that there exists an AR sequence ending in $M$ if and only if there exists an AR sequence starting from $M$.
The appendix of \cite{AKM} contains a nice proof of Theorem~\ref{Yo theorem} in a  setting which includes Gorenstein rings of dimension one. 

\begin{lem} \label{L_p=Q-projective} Assume $\dim R=1$, and let $N \in \CM(R)$. Then $N \in L_p(R)$ if and only if $N \otimes_R Q$ is a projective $Q$-module. \end{lem}

\begin{proof} The prime ideals of $Q$ correspond to $\min R$. Now use the fact that, since $Q$ is noetherian, a $Q$-module is projective precisely when it is free at all maximal ideals of $Q$ (cf. \cite[ Exercise 4.11]{Eisenbud:book}). \qed \end{proof}

\sit \label{sit:Gor 1} For the remainder of this section assume furthermore that $R$ is Gorenstein of dimension one, and $M \ncong R$ is an indecomposable in $ L_p(R)$. Then (ignoring a graded shift, in the graded case; it will not concern us) $\tau(M)=\syz_R(M)$ (cf. \cite[3.11]{Yoshino:book}), where $\syz_R(M)$ denotes the syzygy module of $M$, which is defined to be the kernel of a minimal surjection onto $M$ by a free $R$-module. The module $\tau^{-1}(M)=\syz_R^{-1}(M) \in L_p(R)$ is determined by $\syz_R(\syz_R^{-1}(M)) \cong M$, and can be computed via $\syz_R^{-1}\cong (\syz_R(M^*))^*$.

\defn \label{defn:stable}  Given a  ring $A$, and  $A$-modules $X$ and $Y$, $\underline{\Hom}_A(X,Y)$ denotes \\ $\Hom_A(X,Y)/ \{\text{maps factoring through projective $A$-modules}\}$; and $\stEnd_A(X)$ denotes $\underline{\Hom}_A(X,X)$. An $A$-homomorphism is said to be 
 \emph{stably zero} if it factors through a projective $A$-module.

\begin{lem}\label{lemma:stable hom} \cite[Lemma 3.8]{Yoshino:book} Let $A$ be a commutative ring, and let $X$ and $Y$ be finitely generated $A$-modules. The sequence 
\begin{center} $\xymatrix{\Hom_A(X,A) \otimes_A Y \ar[r]^-{\mu} & \Hom_A(X,Y) \ar[r] &\underline{\Hom}_A(X,Y) \ar[r] &0}$ \end{center}
 is exact, where $\mu \colon \Hom_A(X,A) \otimes_A Y \to \Hom_A(X,Y)$ is given by $f \otimes y \mapsto (x \mapsto f(x)y)$.\end{lem}

\begin{lem} \label{stEnd=Ext} $ \stEnd_R(M) \cong \Ext_R^1(\syz^{-1}_R(M),M) $ as left $\End_R(M)$-modules. \end{lem}
\begin{proof} Let $N=\syz^{-1}_R(M)$. By applying $\Hom_R(\_,M)$ to a short exact sequence $0 \ra M \ra F \ra N \ra 0$ where $F$ is free, we have an exact sequence $ \Hom_R(F,M) \ra \Hom_R(M,M) \ra \Ext_R^1(N,M) \ra \Ext_R^1(F,M)=0$. It only remains to observe that  the image of $\Hom_R(F,M) \ra \Hom_R(M,M)$ consists of all endomorphisms factoring through projectives, which simply follows from the definition of projective. \qed \end{proof}

 \begin{remark}  \label{AR seq as pushout} Left $M \in L_p(R)$ be a nonfree  indecomposable. Then $\End_R M$ is a (graded-) local ring (cf. \cite[proposition 8]{Auslander-Reiten:gradedCM}), and therefore so is $\stEnd_R M$. It follows from Lemma~\ref{stEnd=Ext} and Theorem~\ref{Yo theorem} that the ring  $\stEnd_R(M)$ has a simple socle when considered as a left module over itself, and that if  $h : M \ra M$ generates this socle then the AR sequence starting from $M$ equals the pushout via $h$ of the short exact sequence $0 \ra M \ra F \ra \syz^{-1}_R (M) \ra 0$ where $F$ is free. In particular, if $\iota$ denotes the given injective map $M \ra F$, and $0\ra M \ra X \ra N \ra 0$ is the AR sequence starting from $M$, then $X \cong (M \oplus F)/\{(-h(m),\iota(m))| m \in M\}$.
\end{remark}

\section{Testing stable-vanishing with trace.} \label{sec:trace argument!}

In this section, let $R$ simply be a commutative ring, let $Q=Q(R)$, and let $M$ be a finitely generated $R$-module such that $M \otimes_R Q$ is a projective $Q$-module. Let $(\_)^*$ denote $\Hom_R(\_,R)$.
 
\notation \label{D_B} Given an $R$-algebra $B$, let $D_B(\_)$ denote  $\Hom_R(\_,B)$. Let $\nu_B$ denote  $D_B ((\_)^*)=D_B\circ D_R(\_)$, and let $\lambda_B$ denote the Hom-Tensor adjoint isomorphism $\lambda_B \colon D_B(M^* \otimes_R \_) \to \Hom_R(\_, \nu_B M)$. We also let $\mu_M$ denote the natural transformation  $\mu_M \colon M^* \otimes_R \_ \ra \Hom_R(M, \_ \,)$ given by $f \otimes x \mapsto (m \mapsto f(m)x)$. For future reference, we note that for a given $R$-module $X$, the map  $\lambda_B \circ (D_B \mu_M) \colon D_B \Hom_R(M,X) \to \Hom_R(X,\nu_B M)$ is given by the rule \begin{equation}\label{lambda D mu} [\lambda_B \circ D_B \mu_M](\sigma)(x)(f)=  \sigma(\mu_M(f \otimes x)) \, \text{, for all } \sigma \in D_B \Hom_R(M,X), \,\, x \in X, \,\, f \in M^*. \end{equation}

Let $E=Q/R$. The exact sequence $\xymatrix{0 \ar[r] & R \ar[r]^{\iota} & Q \ar[r]^q & E \ar[r] & 0}$ induces the  exact commutative diagram
\begin{equation} \label {dgm:D nu}
  \begin{gathered}
    \xymatrix{
      0 \ar[r] & D_R \Hom_R(M, \_\,) \ar[r] \ar[d]^{\lambda_R \circ D_R \mu_M}
      &  D_Q \Hom_R(M, \_\,) \ar[r]^{q_*} \ar[d]^{\lambda_Q \circ D_Q\mu_M} &  D_E \Hom_R(M, \_\,) \ar[d]^{\lambda_E \circ D_E\mu_M}  \\
      0 \ar[r] & \Hom_R(\_\,,\nu_R M) \ar[r]^{\iota_*} & \Hom_R(\_\,,\nu_Q M)  \ar[r] & \Hom_R(\_\,,\nu_E M) \\ 
    }
        \end{gathered}.
\end{equation}

     We now show that $D_Q\mu_M$ is an isomorphism on the category of finitely generated $R$-modules, so that the second map in the composable pair
   \begin{equation} \label {half-exact sequence}
  \begin{gathered}
    \xymatrix@C+45pt{    
       D_R \Hom_R(M, \_\,) \ar[r]^{\lambda_R \circ D_R\mu_M} & \Hom_R(\_\,,\nu_R M)  \ar[r]^{q_*  (\lambda_Q \circ D_Q\mu_M)^{-1}  \iota_*}  & D_E \Hom_R(M, \_\,)}
        \end{gathered}
\end{equation}
    is well-defined.

     \begin{lem} \label{lem:proof of these facts} \cite[Appendix]{AKM} 
          \begin{enumerate} \item The map $D_Q\mu_M$ is an isomorphism  on finitely generated $R$-modules, and the sequence~\eqref{half-exact sequence} is exact. \item If $R$ is Gorenstein of dimension one, and both $M$ and the input module $X$ lie in $\CM(R)$, then the image of $\lambda_R \circ D_R\mu_M$ consists of the stably zero maps $X \to \nu_R M$. \end{enumerate} \end{lem}

\begin{proof}   
    We can identify $\mu_M \otimes_R Q$  with $\mu_{M \otimes Q} \colon \Hom_Q(M \otimes_R Q , Q) \otimes_Q (\_ \otimes_R Q) \to \Hom_Q(M \otimes_R Q, \_ \otimes_R Q)$, which is an isomorphism because $M \otimes_R Q$ is a projective $Q$-module. Thus $D_Q\mu_M$ is an isomorphism, since it can be viewed as $D_Q(\mu_M \otimes_R Q)$. The exactness of~\eqref{half-exact sequence} is seen by chasing the diagram~\eqref{dgm:D nu}.
    
    Now we assume the hypotheses of (2). Let $\xymatrix{0 \ar[r] & \syz_R(M) \ar[r]&F\ar[r]^p&M\ar[r]&0}$ be a short exact sequence, where $F$ is a free $R$-module. Consider the commutative diagram \begin{equation} \label{p diagram} \xymatrix@C+10pt{(\Hom_R(F,X))^* \ar[d] \ar[r]&(F^* \otimes_R X)^*\ar[d] \ar[r] &  \Hom_R(X,F^{**}) \ar[d] & \Hom_R(X,F) \ar[d] \ar[l]\\ (\Hom_R(M,X))^* \ar[r]^{D_R \mu_M} & (M^* \otimes_R X)^* \ar[r]^{\lambda_R} & \Hom_R(X,M^{**}) & \Hom_R(X,M)\ar[l] \,\,,}\end{equation}
    where the vertical maps are induced by $p \colon F \to M$, and the horizontal maps on the right are the isomorphisms  induced by $M \cong M^{**}$ and $F \cong F^{**}$. It is easy to see that the image of the rightmost vertical map consists of the stably zero maps $X\to M$, and it follows that the third vertical map consists of the stably zero maps $X \to M^{**}$. Let $H$ denote the map $\Hom_R(M,X) \to \Hom_R(F,X)$ induced by $p$. Since the top row of diagram~\eqref{p diagram} consists of isomorphisms, establishing surjectivity of the leftmost vertical map, namely $D_R H$, is sufficient for proving (2). Let $N=\cok H$. By left-exactness of $\Hom$, we have a left-exact sequence    \begin{center}
    $\xymatrix{0 \ar[r] & \Hom_R(M,X) \ar[r]^H & \Hom_R(F,X) \ar[r] & \Hom_R(\syz_R(M),X)}$, \end{center}
     and therefore $N$ embeds into $\Hom_R(\syz_R(M),X)$. Thus $N \in \CM(R)$, so $\Ext_R^1(N,R)=0$. Therefore the sequence     
    $\xymatrix{0 \ar[r] & N^* \ar[r]& \Hom_R(F,X)^* \ar[r]^{D_R H} & \Hom_R(M,X)^* \ar[r]&0}$
     is exact, so (2) is proved. \qed \end{proof}
    
\begin{lem} \label{trace argument!}  Assume $R$ is Gorenstein of dimension one, and let $M \in \CM(R)$ and $h \in \End_R M$. Then $h$ is stably zero if and only if $\trace(gh \otimes Q) \in R$  for all $g \in \End_R M$. (Recall the definition of trace, Definition~\ref{defn:trace}.) \end{lem} \begin{proof} 

 Let $\eta$ denote the isomorphism $\End_R M \to \Hom_R(M,M^{**})$ induced by $M \cong M^{**}$, and let $\theta$ denote $(\lambda_Q \circ D_Q\mu_M)^{-1} \circ  \iota_*: \Hom_R(M,M^{**}) \to \Hom_R( \End_R M,Q)$.  It follows from Lemma~\ref{lem:proof of these facts} that   $h$ is stably zero if and only if $[\theta (\eta  h)](g) \in R$ for all $g \colon M \to M$. Thus, by the symmetry of trace, it suffices to show that $[\theta (\eta  h)](g)=\trace(hg \otimes Q)$.  Let $\sigma: \End_R M \to Q$ denote the map sending $g \in \End_R M$ to $\trace(hg \otimes Q)$. Thus, we wish to show $\theta(\eta h)=\sigma$; equivalently, $\iota_*(\eta  h)=(\lambda_Q \circ D_Q \mu_M)(\sigma)$. 
 
  Take a   finite collection of maps $\{\phi_i: M \otimes_R Q \to Q\}_i$ and elements $\{e_i\}_i \in M \otimes_R Q$, such that $w=\sum_i \phi_i(w) e_i$ for all $w \in M \otimes_R Q$; thus $\trace(h\otimes Q)=\sum_i \phi_i((h \otimes Q)(e_i))$, as we mentioned in Definition~\ref{defn:trace}.  Given $m \in M$, and $f \in M^*$, let $g$ denote the endomorphism $\mu_M(f \otimes m)$. By equation~\ref{lambda D mu}, $(\lambda_Q \circ D_Q \mu_M)(\sigma )(m)(f)=\sigma(g)=\trace(hg \otimes Q)=\sum_i \phi_i((h\otimes Q)((g \otimes Q)(e_i)))$. Now using firstly that $g \otimes Q$ is given by $w \mapsto (f \otimes Q)(w) m$, and then that $f \otimes Q$ and the $\phi_i$'s have output in $Q$, we have 
 \begin{center}$(\lambda_Q \circ D_Q \mu_M)(\sigma )(m)(f) =\sum_i \phi_i((h\otimes Q) ((f\otimes Q)(e_i)m))=\sum_i (f \otimes Q)(e_i) \phi_i(h (m))$ 
 $=(f\otimes Q)( \sum_i \phi_i (h (m)) e_i )=f(h(m))=\iota_*(\eta  h)(m)(f)$. \end{center} \qed \end{proof}

   \section{Our main result.} \label{sec:main}
     Our goal now is to prove Theorem~\ref{main theorem}, which is really a formula for the AR sequence beginning at $M$. 
     For this section, let $R$ be a one-dimensional  Gorenstein ring which is either a complete local ring or a connected graded ring.  Let $\p$ be a minimal prime of $R$ and let $R'=R/\p$. Set $Q=Q(R)$, $Q'=Q(R')$, and $\m=\m_R$. 

Recall that for a commutative ring $A$, if $M$ and $N$ are finitely generated $A$-submodules of $Q(A)$, and $M$ contains a faithful element $w$, i.e. $\Ann_A(w)=0$, then $\Hom_A(M,N)$ is naturally identified with $(N:_{Q(A)} M)=\{q \in Q(A)| qM \subseteq N\}$. Essentially Theorem~\ref{main theorem} will assert that the AR sequence beginning at any $M$ supported at $\p$ is induced by an element $\gamma \in \End_R \m$ which may be found by the following recipe: Pick $\gamma' \in (R':_{Q'} \J(\Rpb)) \setminus R'$, lift $\gamma'$ from $Q'$ to $\gpt \in Q$, and find $z \in R$ such that 
      $\p =\Ann_R (z)$ and $z \gpt \notin R$; finally, set $\gamma=z \gpt$. We  first show that these steps for finding $\gamma$ are well-defined.

     \notation  Let $\F(R')$ denote $(R':_{Q'} \J(\Rpb))$. 
     
     \begin{lem} \label{gamma' exists} We have 
     $\F(R') \subseteq \Hom_{R'}(\m_{R'},R')$, while $\F(R') \nsubseteq R'$.
       \end{lem}
 
 \begin{proof} 
 Let $D=R'$. As $\m_D \subseteq \J(\Dbar)$, we have $\F(D) \subseteq \Hom_D(\m_D,D)$.  
  In the complete local case $\Dbar$ is a DVR, while in the connected graded case $\Dbar$ is  a polynomial ring over a field by Lemma~\ref{lem:polynomial ring}. Let $\pi$ denote a  generator  for $\m_{\Dbar}$, and let $n$  be the positive integer such that the conductor ideal  $(D:_D \Dbar)$ equals $\pi^n \Dbar$. It is clear that $\pi^{n-1}\Dbar \subseteq \F(D)$ and $\pi^{n-1}\Dbar \nsubseteq D$.
  \qed \end{proof}

 \begin{lem} \label{Q and Q'} We have an $R$-algebra isomorphism $Q/ \p Q \cong Q'$. \end{lem}
 
 \begin{proof} If $x \in R$  is a (homogeneous) nonzerodivisor, we have $R_x=Q$. To see this, it suffices to check that  a given (homogeneous)  nonzerodivisor $y \in R$ becomes a unit in $R_x$. As $Ry$ is $\m$-primary, we have $x^i=ry$ for some $i \ge 1$ and some $r \in R$. Therefore $y$ is a unit in $R_x$; hence $R_x=Q$.  Now $Q/ \p Q=Q \otimes_R R'= R_x \otimes_R R'=R'_x$, which equals $Q'$ by the same argument as above. \qed \end{proof}

     
     \begin{lem} \label{z exists} Let $\gamma' $ be a (homogeneous) element of $\F(R') \setminus R'$ (which exists by Lemma~\ref{gamma' exists}) and let $\gpt$ be a lift of $\gamma'$ to $Q$ (see Lemma~\ref{Q and Q'}). Then there exists (homogeneous) $z \in R$ such that $\p=\Ann_R(z)$ and $\gpt z \notin R$. 
     \end{lem}
     \begin{proof}  Let $\omega$ denote the ideal $\Ann_R (\p)$.  Now $\omega \cong \Hom_R(R',R)$ is, up to a graded shift, a canonical module for $R'$ (\cite[Theorem 3.3.7]{BH} and \cite[proposition 3.6.12]{BH}), and therefore we have $\End_{R'}\omega\cong R'$ (cf. \cite[Theorem 3.3.4]{BH} and the proof of \cite[proposition 3.6.9b]{BH}). 
     We will also use that $\p=\Ann_R(\omega)=\Ann_R(z)$ for each nonzero $z \in \omega$, which is true because all associated primes of $R$ are minimal, so that any ideal strictly larger than $\p$ contains a nonzerodivisor.
     
 Regarding $\omega$ as a subset of $Q$ via $\omega \subset R \subset Q$, suppose that $\gpt \omega \subseteq \omega$. Then the action of $\gpt$ on $\omega$ agrees with the multiplication on $\omega$ by some  $r\in R$, so $\gpt-r \in \Ann_Q(\omega)=\p Q$. But then $\gamma' \in R'$ is a contradiction. So there must exist $z \in \omega$ such that $\gpt z \notin \omega$. As $\Ann_Q (\p) \cap R=\omega$, we thus have $\gpt z \notin R$. 
     \qed \end{proof}

   \begin{lem} \label{Frobenial} For $\gpt \in Q$ and $z \in R$  as in  Lemma~\ref{z exists}, we have $z \gpt \in  \End_R \m$. \end{lem}

\begin{proof}   From Lemma~\ref{gamma' exists} we have $\gamma' \m R' \subseteq  R'$.  Since $z  R \cong R'$, it follows that $\gpt z  \m \subseteq z R$, thus $z \gpt \in \m^*$. It remains to observe that $\m$ has no free direct summand. But any proper direct summand of an ideal  has nonzero annihilator; and if $\m$ were free $R$ would be regular. \qed \end{proof}  
      
We have now shown that the steps for for finding $\gamma$, given in the introduction of this section, are well-defined. We need one more preliminary lemma.
      
      \begin{lem} \label{base change of trace} Let $M \in L_p(R)$, and $h \in \End_R M$. Then $\trace(h \otimes Q')=\trace(h \otimes Q) +\p Q$. \end{lem}
      
\begin{proof}  Take $\{\phi_i: M \otimes_R Q \to Q\}_i$ and  $\{e_i\}_i \in M \otimes_R Q$ such that $w=\sum_i \phi_i(w) e_i$ for all $w \in M \otimes_R Q$. If $\phi_i'=\phi \otimes_R R' \colon M \otimes_R Q' \to Q'$ and $e_i'$ denotes the image of $e_i$ in $M \otimes_R Q'$, then $w'=\sum_i \phi_i'(w') e_i'$ for all $w' \in M \otimes_R Q'$. Now $\trace(h\otimes  Q')  = \sum_i\phi'( (h \otimes Q')(e_i'))=\sum_i \phi((h \otimes Q)(e_i))+\p Q=\trace(h \otimes  Q)+\p Q$.
\qed \end{proof}
      
\notation \label{gamma_M} Let $\gamma' $ be a (homogeneous) element of $\F(R') \setminus R'$, let $\gpt$ be a lift of $\gamma'$ to $Q$ and let $z \in R$ (homogeneous)  such that $\p=\Ann_R(z)$ and $\gpt z \notin R$. (Those steps are well-defined by the above lemmas.) Let $\gamma=z \gpt$. Assume $M \in L_p(R)$ has no free direct  summands. Then there exists no surjection $M \to R$, so $M^*=\Hom_R(M,\m)$, hence $M \cong \Hom_R(M, \m)^*$ is a module over the ring $\End_R \m$.  Therefore $\gamma$ induces an endomorphism of $M$, by Lemma~\ref{Frobenial}. Denote it by $\gamma_M$. Denote by $[\gamma_M]$ the class of $\gamma_M$ in $\stEnd_R M$.

\begin{theorem} \label{main theorem} Assume $M \in L_p(R)$ is a nonfree  indecomposable.    Then, using Notation~\ref{gamma_M}, we have $[\gamma_M] \in \soc(\stEnd_RM)$. Moreover, $[\gamma_M] \neq 0$ provided  $\dim_{R'_{\p}} (M \otimes_{R} R'_{\p})$ is a unit in $R$. Thus in this case  $\gamma_M$ induces the AR sequence beginning at $M$ (see Remark~\ref{AR seq as pushout}). \end{theorem}

\begin{proof} First we show $[\gamma_M] \in \soc(\stEnd_RM)$, which by Lemma~\ref{trace argument!} is equivalent to having $\trace(\gamma h \otimes Q) \in R$ for an arbitrary nonisomorphism $h \colon M \to M$.  
  As $ \End_RM/ (\m \End_R M)$ is an artinian local ring, there exists  some $i \ge 1$ such that  $h^i \in \m \End_RM$, and thus $h^i \otimes_R R' \in \m \End_{R'} (M')$.  So $\trace (h \otimes Q') \in \J(\overline{R'})$, by Lemma~\ref{trace-in-radical}.  
Now using Lemma~\ref{base change of trace},    $\gpt \trace(h \otimes Q)  + \p Q \in \gamma' \J(\Rpb) \subset R + \p Q$, whence $\gamma \trace(h \otimes Q) \in zR + z\p Q=zR \subset R$. 

Now suppose $n:=\dim_{R'_{\p}} (M \otimes_{R} R'_{\p})$ is a unit in $R$.   We have  $\trace(\gamma_M \otimes Q)=\gamma \trace(1_{M \otimes Q})\in \gamma (n +\p Q)$ by Lemma~\ref{base change of trace}, while $\gamma (n +\p Q)=n \gamma$,  since $\gamma \p=\gpt z \p=0$. Since $\gamma \notin R$ by definition, we have $n \gamma \notin R$, and thus $[\gamma_M] \neq 0$ by Lemma~\ref{trace argument!}. \qed \end{proof}

The following lemma, which we use in Section 6, is for locating $\gamma'$ when $R'$ is Gorenstein.
\begin{lem} \label{a gamma crit} Assume $R'$ is Gorenstein and connected graded, and let $g \in \Hom_{R'}(\m_{R'},R')$ be a homogeneous element, and  $a=\deg g$. If $R'_a=0$, then $g \in \F(R') \setminus R'$. \end{lem}

\begin{proof} Let $D=R'$. Since $D_a=0$, $g$ is not an element of $D$, and it remains to show that $g \in \F(D)$. As $D$ is Gorenstein, $\dim_k \Ext_D^1(k,D)=1$, where $k=D_0$. Then applying $(\_)^*=\Hom_D(\_,D)$ to the short exact sequence $0 \ra \m_D \ra D \ra k \ra 0$ yields  $\dim_k(\m_D^*/D)=1$.  By  Lemma~\ref{gamma' exists}, there exists some homogeneous $\gamma' \in\F(D) \setminus D$. Then $\gamma' \in \m_D^* \setminus D$, so $\m_D^*=D+k \gamma'$. Therefore $g=d+d' \gamma'$ for some $d \in D_a$ and $d' \in D$. But $d=0$ since $D_a=0$.   Thus $g \in d' \F(D) \subseteq \F(D)$. \qed \end{proof}

\section {Stable AR quivers.} \label{sec:AR quivers}

 In this section we provide the background for those unfamiliar with stable translation quivers and their tree classes. Confer, e.g., \cite{AKM} and \cite{Benson:1991},  although the meaning of ``valued'' is different in \cite{Benson:1991}. 

\defn  A \emph{quiver} is a directed graph $\Gamma = (\Gamma_0, \Gamma_1)$, where $\Gamma_0$ is the set of vertices and $\Gamma_1$ is the set
of arrows. A \emph{morphism} of quivers $\phi:\Gamma \ra \Gamma'$ is a pair  $(\phi_0:\Gamma_0 \ra \Gamma'_0, \phi_1: \Gamma_1 \ra \Gamma'_1)$ such that $\phi_1$ applied to an arrow $x \ra y$ is an arrow $\phi(x) \ra \phi(y)$.    For $x \in \Gamma_0$, $x^-$ denotes the set $\{y \in \Gamma_0| \exists \text{arrow }y \ra x \text{ in } \Gamma_1\}$; and $x^{+}=\{y \in \Gamma_0| \exists \text{arrow } x \ra y \text{ in } \Gamma_1\}$. $\Gamma$ is \emph{locally finite} if $x^+$ and $x^-$ are finite sets for each $x \in \Gamma_0$. A \emph{loop} is an arrow from a vertex to itself. A \emph{multiple arrow} is a set of at least two arrows from a given vertex to another given vertex.  

A \emph{valued quiver} is a quiver $\Gamma$ together with a map $v \colon \Gamma_1 \ra \ZZ_{\ge 1} \times \ZZ_{\ge 1}$. 
By a \emph{graph} we mean an undirected graph. A \emph{valued graph} is a graph $G$ together with specified integers $d_{xy} \ge 1$ and $d_{yx}\ge 1$ for each edge $x \text{\textemdash} y$.

\begin{defn} \label{defn:stable-translation-quiver} A \emph{stable translation quiver} is a locally finite quiver together with a quiver automorphism $\tau$ called the \emph{translation}, such that:
\begin{itemize}
\item $\Gamma$ has no loops and no multiple arrows.
\item For  $x \in \Gamma_0$, $x^-=\tau(x)^+$. \end{itemize} \end{defn}

 Given a stable translation quiver $(\Gamma,\tau)$ and a map $v : \Gamma_1 \ra \ZZ _{\ge 0} \times \ZZ_{\ge 0}$, the triple $(\Gamma,v,\tau)$ is called a \emph{valued  stable translation quiver} if $v(x \ra y)=(a,b) \Leftrightarrow v(\tau(y) \ra x)=(b,a)$. 
A stable translation quiver is \emph{connected} if it is non-empty and cannot be written as disjoint union of two subquivers each stable under the translation. 

\defn \label{defn:component} Let $C$ be a full subquiver of a  quiver $\Gamma$ which satisfies Definition~\ref{defn:stable-translation-quiver} except possibly for the no-loop condition. We call $C$  a \emph{component} of $\Gamma$  if:

(1) For all vertices $x \in C$, we have  $\tau x \in C$ and $\tau^{-1} x \in C$ .

(2) $C$ is a union of connected components of the underlying undirected graph of $\Gamma$.

(3) There is no proper subquiver of $C$ that satisfies (1) and (2).
 
\defn \label{directed tree} By a \emph{directed tree} we shall mean a quiver $T$, with no loops or multiple arrows, such that the underlying undirected graph of $T$ is a tree, and for each $x\in T$, the set $x^-$ has at most one element.

Given a directed tree $T$, there is an associated stable translation quiver $\ZZ T$ defined as follows. The vertices of $\ZZ T$ are the pairs $(n,x)$ with $n \in \ZZ$ and $x$ a vertex of $T$. The arrows of $\ZZ T$ are determined by the following rules: Given vertices $x, y \in T$, and $n \in \ZZ$,

\begin{itemize}

\item $(n,x) \ra (n,y) \in \ZZ T \Leftrightarrow x \ra y  \in T \Leftrightarrow (n,y) \ra (n-1,x) \in \ZZ T$; 

\item If $n' \notin \{n, n-1\}$, there is no arrow $(n,x) \ra (n',y)$.

\end{itemize}

\begin{remark} \label{valued ZZ T} Let $T$ be a valued quiver which is also a  directed tree. Then there is a unique extension of $v$ to $\ZZ T$ such that the latter becomes a valued stable translation quiver. Namely, if  $v( x \ra y )=(a,b)$, then $v((n,x) \ra (n,y))=(a,b)$, and $v( (n,y) \ra (n-1,x) )=(b,a)$. \end{remark}

\begin{lem} \label{ZZ T determines T} Let $T$ and $T'$ be (valued) directed trees. Then $\ZZ T \cong \ZZ T'$ as (valued) stable translation quivers if and only $T \cong T'$ as (valued) graphs.\end{lem}
\begin{proof} See~\cite[proposition 4.15.3]{Benson:1991}. \qed \end{proof}

A group $\Pi$ of automorphisms (commuting with the translation) of a stable translation quiver $\Gamma$ is said to be \emph{admissible} if no orbit of $\Pi$ on the vertices of $\Gamma$ intersects a set of the form $\{x\} \cup x^+$ or $\{x\} \cup x^-$ in more than one point. 
The quotient quiver $\Gamma/\Pi$, with vertices the $\Pi$-orbits of $\Gamma_0$, and with the induced arrows and translation, is also a  stable translation quiver. A surjective morphism of stable translation quivers $\phi:\Gamma \ra \Gamma'$ is called a \emph{covering} if, for each $x \in \Gamma_0$, the induced maps $x^- \ra \phi(x)^-$ and $x^+ \ra \phi(x)^+$ are bijective. 
Note that if $\Pi$ is an admissible group of automorphisms of $\Gamma$,
\begin{equation}\label{quotient is a covering} \text{the canonical projection } \Gamma \ra \Gamma/ \Pi  \text{ is a covering.}\end{equation}

\begin{theorem} \label{RST} \emph{(Riedtmann Structure Theorem}; see \cite[Theorem 4.15.6]{Benson:1991} \emph{)} Given a connected stable translation quiver $\Gamma$, there is a directed tree $T$ and an admissible group of automorphisms $\Pi \subseteq \Aut (\ZZ T)$ such that $\Gamma \cong \ZZ T/\Pi$. In particular, we have a covering $\ZZ T \ra \Gamma$. The underlying undirected graph of $T$ is uniquely determined by $\Gamma$, up to isomorphism. \end{theorem}

The underlying undirected graph of $T$ is called the \emph{tree class} of $\Gamma$.

\begin{remark}\label{rmk:treeclass} Formally, the tree class $T$ of $\Gamma$ is constructed as follows (as in the proof of Theorem~\ref{RST}, which we will not reproduce here). Choose any vertex $x \in \Gamma$, and define the vertices of $T$ to be the set of paths \begin{center} $(x=y_0 \ra y_1 \ra \cdots \ra y_n) \quad (n \ge 0)$ \end{center}
for which no $y_i=\tau(y_{i+2})$. The arrows of $T$ are \begin{center}
$(x=y_0 \ra y_1 \ra \cdots \ra y_{n-1}) \to (x=y_0 \ra y_1 \ra \cdots \ra y_n)$ .\end{center}
\end{remark}

\begin{remark} \label{rmk:valued tree class} Suppose $\Gamma$ is a valued stable translation quiver, and let $\phi: \ZZ T  \ra \Gamma$ be a covering, which exists by the Theorem. Now $\ZZ T$ becomes a valued stable translation quiver, by setting $v(x\ra y)=v(\phi(x\ra y))$. In particular, $T$ becomes  a valued quiver. \end{remark}

\defn \label{defn:valued tree class} The \emph{valued tree class} of a stable translation quiver $\Gamma $ is a valued graph $(T,v)$ where $T$ denotes the tree class of  $\Gamma $, and $v \colon \{\text{edges of } T\} \ra \ZZ_{\ge 0} \times \ZZ_{\ge 0}$ is given as in Remark~\ref{rmk:valued tree class}. 


\begin{defn} \label{defn:subadditive}
Let $(\Gamma,v)$ be a valued, locally finite  quiver without  multiple arrows. For $x\ra y$ in $\Gamma$, we write $v(x \ra y)=(d_{xy},d_{yx})$. If there is no arrow between $x$ and $y$, we set $d_{xy}=d_{yx}=0$. Let $\QQ_{>0}$ be the set of positive rational numbers.

(i) A \emph{subadditive function} on $(\Gamma,v)$ is a $\QQ_{>0}$-valued function $f$ on the set of vertices of $\Gamma$ such that $2f(x) \ge \sum_{y \in \Gamma} d_{yx} f(y)$, for each vertex $x$.

(ii) An \emph{additive function} on $(\Gamma,v)$ is a $\QQ_{>0}$-valued function $f$ on the set of vertices of $\Gamma$ such that $2f(x) =\sum_{y \in \Gamma} d_{yx} f(y)$, for each vertex $x$.
\end{defn}

%

\begin {lem}\label{lemma:Dynkin diagrams}  \cite[Theorem 4.5.8]{Benson:1991} Let $(\Gamma,v)$ be a  connected valued quiver  without loops or multiple arrows. Suppose $f$ is a subadditive function on $\Gamma$, and assume $\Gamma$ has infinitely many vertices. Then: \begin{enumerate}\item The underlying valued graph of $\Gamma$ is an infinite Dynkin diagram.
 \item If $f$ is unbounded, or if $f$ is not additive, then the underlying valued graph of $\Gamma$ is $A_\infty$. \end{enumerate}\end{lem}

\section{The Cohen-Macaulay setting} The over-arching ideas of this section are largely from \cite{AKM}, which in turn is based partly on \cite{Happel1980}. Since we confine our arguments to the commutative setting, they are sometimes slightly easier. Also, the proof of \cite[Lemma 1.23]{AKM} is flawed, since its penultimate sentence is false, and we give a correct version as Lemma~\ref{no loops}. A principal goal of this section is to provide some sufficient conditions for guaranteeing that a given component of the stable AR quiver of a hypersurface is a tube; see  Proposition~\ref{tube proposition}. 

Assume $R$ is a Cohen-Macaulay complete local ring, with maximal ideal $\m=\m_R$. (But the same results hold when $R$ is connected graded instead of complete local.)

\defn \label{defn:Irr} If $M$ and $N$ are indecomposables in $\CM(R)$,  let $\Irr(M,N)$ denote the module of nonisomorphisms $M \to N$ modulo those which are not irreducible. Let $k_M$ denote the division ring $(\End_R M)/\J(\End_R M)$. Thus $\Irr(M,N)$ is a right $k_M$-space, and a left $k_N$-space.

\defn \label{defn:AR quiver} The  Auslander-Reiten quiver of $R$ is the valued quiver defined as follows:
 \begin{itemize}
\item Vertices are isoclasses of indecomposables in $\CM(R)$.
\item There is an arrow $M \ra N$ if and only if there exists an irreducible morphism $M \ra N$, i.e. $\Irr(M,N) \neq 0$. The value $v(M \ra N) $ of the arrow $M \ra N$ is $(a,b)$ where $a$ is the dimension of $\Irr(M,N)$ as a right $k_M$-space, and $b$ is the dimension of $\Irr(M,N)$ as a left $k_N$-space.
\end{itemize}

Recall that we use $\tau$ to denote the AR-translate (defined at the end of Definition~\ref{defn:AR-seq}). 

\begin{lem} \label{dim over k_M} Let $M$ and $N$ be indecomposables in $L_p(R)$. \begin{enumerate} \item If $0 \ra \tau N \ra E\ra N\ra 0$ is an AR sequence, the number of copies of $M$ appearing in a direct sum decomposition of $E$ equals the dimension of $\Irr(M,N)$ as a right  $k_M$-space.
\item If $0 \ra M \ra E'\ra \tau^{-1}M\ra 0$ is an AR sequence, then the number of copies of $N$ appearing in a direct sum decomposition of $E'$ equals the dimension of $\Irr(M,N)$ as a left $k_N$-space. \end{enumerate}
\end{lem}
\begin{proof}See \cite[Lemmas 5.5 and 5.6]{Yoshino:book}. \qed \end{proof}
\begin{remark} \label{rmk:quiver-symmetry}  Suppose that $k=R/\m$ is algebraically closed. Then in the notation of Lemma~\ref{dim over k_M}, we have $k=k_M=k_N$, and it therefore follows from Lemma~\ref{dim over k_M} that the number of copies of $N$ appearing in a decomposition of $E'$ equals the number of copies of $M$ appearing in a decomposition of $E$.\end{remark}

Notationally, we allow $\tau$ to be a partially-defined morphism on the AR quiver of $R$; $\tau x$ is defined precisely when the vertex $x$ corresponds to a non-projective module in $L_p(R)$, by \cite[Theorem 3.4]{Yoshino:book}. The following  fact is used in \cite{AKM}, and the proof essentially follows that of  \cite[VII 1.4]{AuslanderReitenSmalo}.

\begin{lem} \label{lem:valued translation} Let $x \ra y$ be an arrow in the AR quiver of $R$, and let $(a,b)=v(x \ra y)$. If $\tau y $ is defined, then $v(\tau y \ra x) =(b,a)$. If $\tau x$ and $\tau y$ are both defined, then $v(\tau x \ra \tau y)=v(x \ra y)$. \end{lem}

\begin{proof} We need not prove the last sentence, as it follows from the previous. Let $M$ and $N \in \CM(R)$ be the modules corresponding to $x$ and $y$ respectively. We first show $k_N$ and $k_{\tau N}$ are isomorphic $k$-algebras, where $k=R/\m$. Let $ \xymatrix@C-6pt{ 0 \ar[r] & \tau N \ar[r]^p & E \ar[r]^q & N \ar[r] & 0 }$ be an AR sequence. Given $h \in \End_R N$, there exists a commutative diagram
\begin {equation} \label{endo of AR seq} \xymatrix{ 0 \ar[r] & \tau N \ar[r]^p \ar[d]_{h'}& E \ar[r]^q  \ar[d]& N \ar[r] \ar[d]^h & 0 \\
0 \ar[r] & \tau N \ar[r]^p & E \ar[r]^q & N \ar[r] & 0 &.}\end{equation}
 Indeed, note that $hq$ is not a split epimorphism, because if $h$ is surjective, then $h$ is an isomorphism, and thus $hq$ is not a split epimorphism because $q$ is not. Therefore, by Definition~\ref{defn:AR-seq}, there exists $u: E\to E$ such that $hq=qu$, and the existence of $h'$ follows. 
 
 By the dual argument, any given $h' \in \End_R(\tau N)$ can be fit into a similar commutative diagram. 

We wish to show that $h \mapsto h'$ induces a well-defined map $k_N \ra k_{\tau N}$. If so then it is  a surjective ring map  from a division ring, hence an isomorphism, so we will be done. It suffices to show that, given any commutative diagram~\eqref{endo of AR seq} such that $h$ is a nonisomorphism, it follows that $h'$ is also a  nonisomorphism. Suppose, to the contrary, that $h$ is an nonisomorphism and $h'$ is an isomorphism. We may assume $h'$ is the identity map, since we could certainly compose the diagram~\eqref{endo of AR seq} with a similar diagram which has $(h')^{-1}$ on the left. As $h$ is not a split epimorphism, it factors through $q$. But then the top sequence in~\eqref{endo of AR seq} splits, cf. \cite[Ch. III, Lemma 3.3]{MacLane:1963}; and this of course is a contradiction. Thus $k_N \cong k_{\tau N}$ as $k$-algebras. 

In particular, $\dim_k(k_N)=\dim_k(k_{\tau N})$. As $\dim_{k_M} \Irr(M,N)=\dim_{k_M} \Irr(\tau N,M)$ is an immediate consequence of Lemma~\ref{dim over k_M}, our aim is to show $\dim_{k_N} \Irr(M,N)=\dim_{k_{\tau N}} \Irr(\tau N,M)$. By the former, we have  $\dim_{k} \Irr(M,N)=\dim_{k} \Irr(\tau N,M)$. Thus,  
$\dim_{k_N} \Irr(M,N)= \dim_k \Irr(M,N)/\dim_k(k_N) =\dim_{k}\Irr(\tau N,M)/\dim_k(k_{\tau N})=\dim_{k_{\tau N}} \Irr(\tau N,M)$. \qed \end{proof}

\defn \label{defn:stable quiver} If $R$ is Gorenstein,  the \emph{stable Auslander-Reiten (AR) quiver} of $R$ is the valued quiver defined as in Definition~\ref{defn:AR quiver}, except that the vertices are only the isoclasses of nonfree indecomposable modules $M \in L_p(R)$. By a \emph{stable AR component}, we shall mean a component (Definition~\ref{defn:component}) of the stable AR quiver.

\begin{defn} \label{defn:periodic} Let $(\Gamma,\tau)$ be a translation quiver, and $x$ a vertex of $\Gamma$. If  $x =\tau^n x$ for some $n >0$, we say that $x$ is $\tau$-\emph{periodic}. A module $M \in \CM(R)$ is said to be $\tau$-periodic if it corresponds to a $\tau$-periodic vertex in the AR quiver of $R$, i.e., $M \cong \tau^n M$. When $R$ is Gorenstein of dimension one, we will omit the prefix ``$\tau$-'' and just say $M$ is periodic. \end{defn}

The following is well-known.

\begin{lem} \label{lem:periodic} If $(\Gamma,\tau)$ is a connected  translation quiver containing a $\tau$-periodic vertex, then all of its vertices are $\tau$-periodic. \end{lem}
\begin{proof}  If $x$ is a vertex in $\Gamma$ and $\tau^n x=x$, then  $\tau^n$ induces a permutation on the finite set $x^-$, and so some power of $\tau^n$ stabilizes $x^-$ pointwise. Thus each vertex in $x^-$ is $\tau$-periodic; likewise for $x^+$, so every vertex in $\Gamma$ is $\tau$-periodic by induction. \qed \end{proof}

\begin{defn} \label{defn:periodic-component} We say that a connected translation quiver is \emph{periodic} if one, equivalently all, of its vertices is $\tau$-periodic. \end{defn}

\begin{defn} \label{defn:tube} A valued stable translation quiver $\Gamma$ is called a \emph{tube} if $\Gamma \cong \ZZ A_\infty /\langle \tau^n \rangle$ for some $n \ge 1$. If $n=1$, $\Gamma$ is called a \emph{homogeneous} tube.\end{defn}

\begin{remark} \label{rmk:tube} Let $\Gamma$ be a connected periodic stable translation quiver, and suppose the valued tree class of $\Gamma$ is $A_\infty$. Then $\Gamma$ is a tube. To see this, let $\Pi$ be an admissible group of automorphisms of $\ZZ A_\infty$ such that $\Gamma \cong \ZZ A_\infty/\Pi$. Note that every automorphism of the stable translation quiver $\ZZ A_\infty$ is of the form $\tau^n$ for some $n \ge 0$. Thus $\Pi=\langle \tau^n \rangle$ for some $n \ge 0$; and the periodicity implies $n \ge 1$. \end{remark}

\notation \label{push} If $R$ is Gorenstein of dimension one, and $M $ is  an indecomposable in $L_p(R)$, define an $R$-module $\push(M)$ as follows. If $M$ is free, let $\push(M)=0$. Otherwise let $\push(M)$ denote the unique module (up to isomorphism) such that there exists an AR sequence 
$0 \to  M \to  \push (M) \to  \syz^{-1}_R(M) \to  0$. More generally, if $M =\bigoplus_{i=1}^n M_i$ with each $M_i$ in $L_p(R)$, then we set $\push(M)=\bigoplus_{i=1}^n \push(M_i)$.

\begin{notation} \label{multiplicity} (See, e.g.,  \cite[14.1-14.6]{Matsumura}.) For an $R$-module $M$, let $e(M)$ denote the multiplicity of $M$. This can  be defined as $e(M)=\lim_{n \rightarrow \infty} \frac{d!}{n^d} \length(M/\m^n M)$, where $d=\dim R$, but the reader may  ignore this definition; we use only the following properties:

\begin{itemize}
\item  If $0 \to M' \to M \to M'' \to 0$ is exact, then $e(M)=e(M')+e(M'')$.

\item For all $M \in \CM(R)$, $e(M)$ is a positive integer.

 \end{itemize} \end{notation}

\begin{notation} \label{e avg}  Define a function  $e_{\avg}$ from  $\tau$-periodic maximal Cohen-Macaulay $R$-modules  to $\QQ_{>0}$ as follows: If $M$ is $\tau$-periodic of period $n$, let $e_{\avg}(M)=\frac{1}{n}\sum_{i=0}^{n-1} e(\tau^i(M))$. \end{notation}

\begin{lem} \label{lem:e_avg} Assume $R$ is Gorenstein of dimension one, and $M \in L_p(R)$ is indecomposable and periodic. If $\push M=X \oplus F$ where $X$ has no free direct summands and $F$ is a (possibly zero) free module, then $X$ is periodic, and $e_{\avg}(\push M) \le 2 e_{\avg}(M)$. \end{lem}

\begin{proof} We know $X$ is periodic from Lemma~\ref{lem:periodic}. Note that if $N \in \CM(R)$ is periodic, then for any $j \in \ZZ$, and $n \in \NN$  a  multiple of the period of $N$,  $\sum_{i=j}^{n+j-1} e(\tau^i N)=n e_{\avg}(N)$.  For each integer $i$, we have by Lemma~\ref{lem:valued translation} an AR sequence $0 \to  \tau^{i+1} M \to  F_i\oplus \tau^i X \to  \tau^i M \to  0$,  where $F_i$ is a (possibly zero) free module. So  $e(\tau^i X) \le e(\tau^{i+1}M)+e(\tau^i M)$, hence $\sum_{i=1}^n e(\tau^i X) \le \sum_{i=1}^n e(\tau^{i+1}M)+\sum_{i=1}^n e(\tau^i M)$ for each $n \in \NN$. This inequality gives the desired result by  taking  $n$ to be a common multiple of the periods of $M$ and $X$, and  dividing  both sides by $n$. \qed \end{proof}

The following goes back at least to \cite{Happel1980} (in a slightly different setting).

\begin{lem} \label{lemma:subadditive-for-periodic} Let $C$ be a connected $\tau$-periodic valued stable translation quiver which is a subquiver of the stable AR quiver of $R$. Then the valued tree class of $C$ admits a subadditive function (Definition~\ref{defn:subadditive}). \end{lem}

\begin{proof} Let $T$ denote the valued tree class  (Definition~\ref{defn:valued tree class}) of $\Gamma$. By definition of $T$, we have a value-preserving covering $\phi \colon \ZZ T \ra C$.  Define a function $f \colon \ZZ T \ra \QQ_{>0}$ by the rule $f(x)=e_{\avg}(\phi(x))$. 
We claim that the restriction of $f$ to $T$ is a subadditive function.  That is,   $2f(x) \ge  \sum_{y \in T } d_{yx} f(y)$, for each vertex $x$ of $T$. By Lemma~\ref{lem:valued translation}, $d_{yx}=d_{(\tau^{-1} y)x}$ for all $x, y \in C$, hence for all $x, y \in \ZZ T$. In what follows, for any $x \in T$, the sets $x^-$ and $x^+$ will always be taken  with respect to $\ZZ T$;  to signify the predecessors of $x$ with respect to $T$ we can use $x^- \cap T$.  If $x \in T$,  then  $x^+$ equals the disjoint union of $x^+ \cap T$ and $\tau^{-1}(x^- \cap T)$. Now, we have \begin{center}$\displaystyle \sum_{y \in T } d_{yx} f(y)=\displaystyle \sum_{y \in x^- \cap T} d_{yx} f(y) + \sum_{y \in x^+ \cap T} d_{yx} f(y)$  $=\displaystyle \sum_{y \in x^- \cap T} d_{\tau^{-1}yx} f(\tau^{-1}y) + \sum_{y \in x^+ \cap T} d_{yx} f( y)=\sum_{y\in x^+} d_{yx} f(y)$. \end{center} So subadditivity of $f$ is equivalent to $2f(x) \ge \sum_{y\in x^+} d_{yx} f(y)$. Since $\phi$ is a covering, $\sum_{y\in x^+} d_{yx} f(y)=\sum_{y \in \phi(x)^+} d_{y \phi(x)} e_{\avg}(y)$, which is  bounded by $2 e_{\avg}(\phi(x))$ by Lemma~\ref{lem:e_avg}. So $f$ is subadditive. \qed \end{proof}

\begin{lem} \label{no loops} Assume $R$ is Gorenstein, let $M \in L_p(R)$ be a nonfree indecomposable, and suppose there exists an irreducible map from $M$ to itself. Let $C$ denote  the  component of the stable AR quiver containing $M$, and assume $C$ is infinite. Then $C$ is a homogeneous tube with a loop at the end:

\begin{tikzpicture} \node[above] at (1,0) (1){$M=X_0$};
\draw[->](2,.4)--(3,.4);
\draw[->](3,.2)--(2,.2);

\node[above] at (3.4,0){$X_1$};

\draw[->](3.8,.4)--(4.8,.4);
\draw[->](4.8,.2)--(3.8,.2);

\node[above] at (5.2,0){$X_2$};

\draw[->](5.6,.4)--(6.6,.4);
\draw[->](6.6,.2)--(5.6,.2);

\node[above] at (7,0){$X_3$};

\draw[->](7.4,.4)--(8.4,.4);
\draw[->](8.4,.2)--(7.4,.2);

\node[above] at (8.8,0){$\cdots$};

\path
(1) edge [loop left] node {} (1);
\end{tikzpicture} 

In particular, $\tau X_i \cong X_i$ for all $X_i \in C$. 
\end{lem}

\begin{proof}  First we show that $M \cong \tau M$. If not, then the AR sequence ending in $M$ is $0 \to \tau M \to M \oplus \tau M \oplus N \to M \to 0$ for some $N \in \CM(R)$. Then $e(N)=0$, hence  $N=0$. Now  Miyata's Theorem \cite[Theorem 1]{Miyata} says that the given AR sequence splits, which is a contradiction. So $\tau M \cong M$. 

 Since $C$ has a loop, it  does not satisfy the definition of stable translation quiver (Definition~\ref{defn:stable-translation-quiver}). But removing the loops in $C$  (and keeping all vertices and all non-loop arrows), we get a $\tau$-periodic connected stable translation quiver; call it $\Gamma$, and let $T$ denote valued tree class of $\Gamma$. Now $T$ admits a subadditive function given by $e_{\avg}$, as in the proof of~\ref{lemma:subadditive-for-periodic}. From the fact that $\Gamma$ is not a full subquiver of the AR quiver of $R$, it follows that  $e_{\avg}$ is strictly subadditive (i.e., not additive). As $\Gamma$ is infinite and $\tau$-periodic, $T$ must be infinite. Therefore $T \cong A_{\infty}$ by Lemma~\ref{lemma:Dynkin diagrams}, and $\Gamma \cong A_{\infty}/\langle \tau \rangle$, by Remark~\ref{rmk:tube}. So $\Gamma$ has the form \begin{center}
\begin{tikzpicture} \node[above] at (1.6,0) (1){$X_0$};
\draw[->](2,.4)--(3,.4);
\draw[->](3,.2)--(2,.2);

\node[above] at (3.4,0){$X_1$};

\draw[->](3.8,.4)--(4.8,.4);
\draw[->](4.8,.2)--(3.8,.2);

\node[above] at (5.2,0){$X_2$};

\draw[->](5.6,.4)--(6.6,.4);
\draw[->](6.6,.2)--(5.6,.2);

\node[above] at (7,0){$X_3$};

\draw[->](7.4,.4)--(8.4,.4);
\draw[->](8.4,.2)--(7.4,.2);

\node[above] at (8.8,0){$\cdots$};

\end{tikzpicture}. \end{center} Suppose $M=X_i$ for some $i>0$. Then we have an AR sequence $0 \to X_i \ra X_i \oplus X_{i-1} \oplus X_{i+1} \oplus F \ra X_i \ra 0$, for some free module $F$,   so $e(X_i) \ge e(X_{i-1})+e(X_{i+1})$. But the AR sequences ending in $X_{i-1}$ and $X_{i+1}$ give us $2  e(X_{i+1} ) > e(X_{i})$ and  $2e(X_{i-1}) \ge e(X_{i})$. These inequalities contradict the previous one, so $M=X_0$. \qed
\end{proof}

The following ``Maranda-type result'' corresponds to Lemma 1.24 in \cite{AKM}. In our  setting, namely that of a Cohen-Macaulay complete local ring, this result is well-known (but possibly has only been stated for the case when the ring is an isolated singularity). The following proof can be found, for example, in \cite[proposition 15.8]{Leuschke-Wiegand:BOOK} and its corollaries.  

\begin{lem} \label{maranda}  Let $M$ and $N$ be nonisomorphic indecomposables in $ L_p(R)$, and let $x\in \m$ be a nonzerodivisor. Then there exists $i \ge 1$ such that  $M/x^iM$ and $N/x^i N$ are nonisomorphic indecomposable modules.\end{lem}

\begin{proof}  Since $M$ lies in  $ L_p(R)$, $\Ext^1_R(M,N)$ has finite length (since for any nonmaximal prime $\p$, we have $0=\Ext^1_{R_{\p}}(M_{\p},N_{\p})=\Ext^1_R(M,N)_{\p}$). Therefore we may assume, after replacing $x$ by a suitable power of itself, that $x \Ext^1_R(M,N)=0$. By applying $\Hom_R(M, \_)$ to the commutative exact diagram
\begin{center}
$
\begin{gathered} \xymatrix {0 \ar[r] & N \ar[r]^{x^2} \ar[d] _x & N \ar[r] \ar@{=}[d] & N/x^2N \ar[r] \ar[d] & 0\\
0 \ar[r] & N \ar[r]^{x} & N \ar[r]& N/xN \ar[r] & 0\\}
\end{gathered}$, \end{center}
we obtain a commutative exact diagram
\begin{equation} \label{maranda diagram}
\begin{gathered} \xymatrix {\Hom_R(M,N) \ar[r] \ar@{=}[d] &\Hom_R(M,N/x^2M) \ar[r] \ar[d] & \Ext^1_R(M,N) \ar[d]^x\\
\Hom_R(M,N) \ar[r] & \Hom_R(M,N/xM) \ar[r] & \Ext^1_R(M,N)\\
} \end{gathered} . \end{equation}

Consider the maps $\theta \colon \Hom_R(M,N) \ra \Hom_R(M/xM,N/xN)$ and $\theta_2 \colon \Hom_R(M/x^2M,N/x^2N) \ra \Hom_R(M/xM, N/xN)$ given by tensoring all maps with $R/(x)$. Notice that in diagram~\eqref{maranda diagram}, the horizontal and vertical maps into $\Hom_R(M,N/xM)$ can be identified with $\theta$ and $\theta_2$ respectively, while the rightmost vertical map is zero. Therefore a diagram chase  yields \begin{equation} \label{eqn:theta_2}  \im(\theta)=\im(\theta_2). \end{equation}
We claim $i=2$ will suffice. Suppose $M/x^2M$ is not indecomposable. Then there exists a nontrivial idempotent $e \in  \End_R(M/x^2M)$. Consider the equation \eqref{eqn:theta_2} in the case $M=N$; now $\theta$ and $\theta_2$ are of course ring homomorphisms. Since $\End_R M$ is (noncommutative-) local, so is $\im \theta$, and therefore  $\theta_2(e)$ is either 0 or 1. Since $1-e \in \End_R(M/x^2M)$ is also a nontrivial idempotent, we may assume $\theta_2(e)=0$, i.e. $\im e \subseteq xM/x^2M$. But then $e^2=0$ is a contradiction.

Now suppose $\varphi \colon M/x^2M \to N/x^2N$ is an isomorphism, with inverse $\psi \colon N/x^2N \to M/x^2M$. By~\eqref {eqn:theta_2}, there exist $\tilde{\varphi} \colon M \to N$ such that $\tilde{\varphi} \otimes_R (R/x)=\varphi \otimes_R (R/x)$, and $\tilde{\psi} \colon M \to N$ such that $\tilde{\psi} \otimes_R (R/x)=\psi \otimes_R (R/x)$. By  Nakayama's Lemma, $\tilde{\varphi}$ and $\tilde{\psi}$ are surjective. Thus  $\tilde{\psi} \tilde{\varphi}$ is a surjective endomorphism, equivalently, an isomorphism; and thus $\tilde{\varphi}$ is an isomorphism. \qed \end{proof}

\begin{lem} \label{maps involving R} Assume $\dim R=1$, and  let $M$ be an  indecomposable in $\CM(R)$. Then there exists an irreducible morphism $M \to R$ if and only if $M$ is isomorphic to a direct summand of $\m$. If $R$ is Gorenstein, there exists an irreducible morphism $R \to M$ if and only if $M$ is isomorphic to a direct summand of $\m^*$. \end{lem}

\begin{proof} Write $\m$ as a direct sum of indecomposables, $\m=\bigoplus_i \m_i$. Let $\iota_i$ denote the  inclusion map $\m_i \into R$. To see that $\iota_i$ is irreducible, take a factorization $\iota_i=hg$ in $\CM(R)$, where $h$ is not a split epimorphism. Then $h$ is not onto, so $\im h \subseteq \m$. Then if $h'$ denotes the map into $\m$ given by $x \mapsto h(x)$, and $p_i$ denotes the projection $\m \onto \m_i$, we have that $p_ih'g = 1_{\m_i}$, so $g$ is a split monomorphism; hence $\iota_i$ is irreducible. Now let $M$ be an indecomposable in $\CM(R)$ and let $f: M \to R$ be an irreducible morphism. Let $\iota$ denote the inclusion map $\m \into R$. Since $f$ is not a split epimorphism, $\im f \subseteq \m$,  hence $f= \iota g$ for some $g:M \to \m$. As $\iota$ is certainly not a split epimorphism, $g$ is a split monomorphism. 

For the last sentence of the statement, note that the irreducible maps from $R$ are obtained by dualizing the irreducible maps into $R$. \qed \end{proof}

We recall the Harada-Sai Lemma:
\begin{lem}\label{H-S}  \cite[VI. Cor. 1.3]{AuslanderReitenSmalo}
Let $\Lambda$ be an artin algebra (e.g. a commutative artinian ring). If $f_i \colon M_i \ra M_{i+1}$ are nonisomorphisms between indecomposable modules $M_i$ for $i=1,...,2^n-1$ and $\length(M_i) \le n$ for all $i$, then $f_{2^n-1} \cdots f_1=0$. \end{lem}

\begin{lem} \label{infinite components}  \cite[proposition 1.26]{AKM}  Assume $R$ is Gorenstein of dimension one, and $\m$ is indecomposable; and suppose $R$ has a stable AR component $C$ which is finite. Then $C$ consists of all isoclasses of non-projective indecomposables in $\CM(R)$. \end{lem}

\begin{proof} As $C$ is finite, Lemma~\ref{maranda} implies that we can take $x \in \m$ such that for each pair $M \ncong N$ in $C$, $M/xM$ and $N/xN$ are nonisomorphic indecomposable modules.

We may assume $R$ is not regular, and therefore $\m$ is not free. Now first we show $\m \in C$. Suppose not; then there are no irreducible maps to $R$ from any module in $C$   (Lemma~\ref{maps involving R}). Therefore if $N \in C$ and $ N \ra N'$ is any irreducible map in $\CM(R)$,  $N'$ must lie in $C$ (since $L_p(R)$ is closed under $\syz_R^{-1}$, and therefore under irreducible maps by consideration of AR sequences). Pick a module $M \in C$. By replacing $x$ by a power of itself if necessary, we can choose $f \colon M \ra R$ such that $f(M) \nsubseteq xR$, i.e. $f\otimes_R (R/x)\neq0$. Since $f$ is not a split monomorphism, and there exists an AR sequence beginning in $M$,  $f$ equals a sum of maps of the form $gh$, where $h$ is an irreducible map between modules in $C$. Since $g \in \Hom_R(N,R)$ for some $N \in C$, $g$ is not a split monomorphism, and can in turn be written as a sum of maps of the form $kl$ where $l$ is  an irreducible map  in $C$; now $f=\sum klh$. Continue this process until we have written $f$ as a sum $\sum_i g_i h_{2^{n-1},i} \cdots h_{1,i} $ where each $h_{j,i}$ is an irreducible map in $C$, and $n=\max\{\length(N/xN)|N \in C\}$. Note that each $h_{j,i} \otimes_R (R/x)$ is a  nonisomorphism by our assumption on $x$ together with Lemma~\ref{no loops}. Therefore,  Lemma~\ref{H-S} implies $f\otimes_R (R/x)=0$, contradiction. Thus $\m \in C$.

Now just suppose $C$ omits some indecomposable nonfree $M \in \CM(R)$. Again choose $f \colon M \ra R$ such that $f\otimes_R (R/x)\neq0$. Note that any map to $R$ which is not a split epimorphism factors through $\m$. Whereas in the previous paragraph we reached a contradiction via Lemma~\ref{H-S}, by ``stacking irreducible maps while moving forwards through $C$'',  we now obtain a contradiction by ``stacking irreducible maps while moving backwards through $C \cup \{R\}$''. \qed \end{proof}

\begin{remark} \label{rmk:push} Assume $R$ is Gorenstein and let $C$ be a stable AR component without loops. Then $C$ is a valued stable translation quiver (by  Lemma~\ref{lem:valued translation}) and therefore has a valued tree class $T$ (Definition~\ref{defn:valued tree class}). Then $T$ carries the information of how many nonfree direct  summands  $\push(M)$ and $\push(\push(M))$ (in general, $\push^i(M)$)  have for modules $M \in C$. Let us explain further. Let $x$ be the vertex in $C$ corresponding to $M$, and  let $n=\displaystyle  \sum _{(x \ra y) \in C} d_{yx}$. Then $n$ is the number of nonfree summands in $\push(M)$; that is, $\push(M)=F \oplus \bigoplus_{i=1}^n X_i$ where $F$ is a (possibly zero) free module, and  the $X_i$ are (not necessarily nonisomorphic) nonfree indecomposables in $L_p(R)$. We have a  value-preserving covering  $\phi\colon \ZZ T \ra C$, and after possibly composing $\phi$ with a power of $\tau$, we have $x \in \phi(T)$, say $x=\phi(u)$. Since $\phi\colon \ZZ T \ra C$ is a covering, 
$\displaystyle  \sum _{(x \ra y) \in C} d_{yx} =\sum_{(u \ra w) \in \ZZ T} d_{wu}$, and by definition of $\ZZ T$ this equals $\displaystyle \sum_{w \in T  } d_{wu}$. Thus $n=\displaystyle \sum_{w \in T  } d_{wu}$. Likewise, $\displaystyle \sum_{w,z \in T } d_{zw} d_{wu}$ is the number of nonfree direct summands in $\push(\push(M))$.
\end{remark}

\begin{prop} \label{tube proposition} (cf. \cite[Lemma 1.23 and Theorem 1.27]{AKM}) Assume  $R$ is Gorenstein of dimension one, $\m$ is indecomposable, and $\CM(R)$ has infinitely many isoclasses of indecomposables. Let $C$ be a periodic component of the stable AR quiver of $R$, and suppose that either $R$ is a reduced hypersurface and $C$ has no loops, or that there exists some $M \in C$ such that $\push(\push(M))=X \oplus Y \oplus F$ for some indecomposables $X$ and $Y$, and some possibly-zero free module $F$.
 Then, $C$ is a tube.\end{prop}

\begin{proof}  If $C$ has a loop, then by Lemma~\ref{no loops}, for every $M \in C$, the module $\push M$ has two nonfree indecomposable summands, and therefore $\push(\push(M))$ has four. So we may assume $C$ has no loops. Thus $C$ is a valued stable translation quiver, and we have a valued directed tree $T$ and a value-preserving covering $\phi \colon \ZZ T \ra C$.  Let the function $f \colon \ZZ T \ra \QQ_{>0}$ be given by $f(x)=e_{\avg}(\phi(x))$. As seen in Lemma~\ref{lemma:subadditive-for-periodic}, $f$ restricts to a subadditive function on $T$. 
Since $\phi$ is surjective, every vertex of $C$ lies in the $\tau$-orbit of a vertex in $\phi(T)$. Note also that $C$ has infinitely many vertices, by Lemma~\ref{infinite components}. Therefore $T$ is infinite, so it is an infinite Dynkin diagram by Lemma~\ref{lemma:Dynkin diagrams}. If  $R$ is a reduced then $\{e(M)| M \in C\}$ is unbounded (see \cite[Theorem 6.2]{Yoshino:book}); and so if $R$ is a reduced hypersurface (and thus all modules in $C$ have period 2) then $f$ is unbounded. Then  $T \cong A_\infty$, by Lemma~\ref{lemma:Dynkin diagrams}. If the alternate condition holds, we get $T\cong A_{\infty}$  by eliminating the  other infinite Dynkin diagrams, in light of Remark~\ref{rmk:push}.   Thus $C$ is a tube, by Remark~\ref{rmk:tube}. \qed \end{proof}

\section{An example.} \label{sec:eg}
In this section, we apply the results of the previous sections to determine the shape (namely, a tube) of some  components of the Auslander-Reiten quiver of the ring  $\hat{R}$ defined in~\ref{R setup}, below. Recall that a hypersurface, i.e. a regular (graded-) local ring modulo a nonzerodivisor, is always Gorenstein. 

\sit \label{sit:S and R} Let $S$ be a regular (graded-) local ring, and $f \in S$ a nonzero element. Let $R=S/fS$. A \emph{matrix factorization} of $f$ is a pair of matrices $(\varphi, \psi)$, with entries in $S$, such that $\varphi \psi=\psi \varphi=f \id_{l \times l} $ for some $l >0$.  As consequences of the definition, we have $ \cok \varphi \cong \cok (\varphi \otimes_S R)$, and (\cite[7.2.2]{Yoshino:book})
\begin{equation} \label{hypersurface resolution} \im (\varphi \otimes_S R )=\ker (\psi \otimes _S R) \,\, \text{ and }  \im (\psi \otimes_S R )=\ker (\varphi \otimes _S R). \end{equation} 
In particular, $\cok \varphi$ and $\cok \psi$ are periodic $R$-modules, of period two. 

\begin{remark} \label{MF morphism} Let $(\varphi, \psi)$ and $(\varphi',\psi')$ be matrix factorizations of $f$. Let $n_1$ and $n_2$ be the integers such that $\varphi$ is $n_1$-by-$n_1$ and $\varphi'$ is $n_2$-by-$n_2$. Given $h \colon \cok \varphi \to \cok \varphi'$, there of course exist $\alpha \colon S^{(n_1)} \to S^{(n_2)}$ and $\beta \colon S^{(n_1)} \to S^{(n_2)}$ making the diagram 
\begin{equation} \label{dgm:alpha beta} \xymatrix{S^{(n_1)} \ar[r]^{\varphi} \ar[d]_{\beta} & S^{(n_1)} \ar[r] \ar[d]^{\alpha} & \cok \varphi \ar[r] \ar[d]^h &0\\
S^{(n_2)} \ar[r]^{\varphi'}& S^{(n_2)} \ar[r] & \cok \varphi' \ar[r] &0} \end{equation}
commute. Now it is easy to see that $\Big( \begin{pmatrix} \varphi' & -\alpha\\ 0& \psi \end{pmatrix}, \begin{pmatrix} \psi' & \beta\\0& \varphi \end{pmatrix} \Big)$ is a matrix factorization of $f$.

 If $(\varphi,\psi)$ is a matrix factorization such that $\varphi$ and $\psi$ each contains no unit entry, then it is called a \emph{reduced} matrix factorization. 
If $(\varphi,\psi)$ is a reduced matrix factorization, then neither $\im \varphi$ nor $\im \psi$ contains a free summand (cf. \cite[7.5.1]{Yoshino:book}). \end{remark}


\sit \label{gamma pushout} Let $(\varphi, \psi)$ be a reduced matrix factorization, let $M=\cok \varphi$, and assume that $R$, $M$ and $\gamma$ satisfy the hypotheses of Theorem~\ref{main theorem}.  Pick $\alpha$ and $\beta$ lifting $\gamma_M \in \End_R M$ in the sense of Remark~\ref{MF morphism}. One may check that the valid choices for $\alpha$ are precisely those choices such that  $\psi \alpha= \gamma \psi$ after passing to $R$. By Remark~\ref{AR seq as pushout},  $\push(M) \cong (\im (\psi\otimes_S R) \oplus R^{(n)})/\{(-\gamma c, c)| c\in \im (\psi\otimes_S R)\}$, where $n$ denotes the side length of the matrices $\varphi$ and $\psi$. Then we see that $\push(M)  = \cok \begin{pmatrix}\varphi & -\alpha\\0 & \psi\\ \end{pmatrix}$ .

\sit \label{R setup}  Let $k$ be a field, of characteristic not equal to 2, and let us set up a connected graded hypersurface $R$ as follows. Let $p$ and $q$ be \emph{relatively prime} integers $\ge 3$, and let $S=k[x,y]$ be the graded polynomial ring such that $S_0=k$, $\deg x =q$, and $\deg y=p$. Let $f \in S$ be a  homogeneous polynomial which is not divisible by $x$. Let $g=(bx^p+y^q)f$, where $b \in k$, and $b$ is allowed to be zero. Now, let $R=k[x,y]/(g)$. The $\m$-adic completion of $R$  is  $\hat{R}=k[|x,y|]/(g)$. Let $v=\deg(f)/p$, which is an integer because $x \nmid f$. We assume  that $f  -y^v\in xS$. Lastly, assume that there are infinitely many isoclasses of indecomposables in $\CM(R)$. 

 Now fix an ideal of $R$ of the form $I =(x^m, y^n)$, where $1\le m<p-1$ and $2\le n<q$.  
 We will show that the component of $\Gamma_{\Rhat}$ containing $\hat{I}$ is a tube, by showing that $\push(\push(\hat{I}))$ has only two indecomposable summands, and applying Proposition~\ref{tube proposition}. However, we will work over $R$:
 
 \begin{remark} \label{Gamma: R and Rhat} Let $C$ be a component of $\Gamma_R$. Now consider the valued translation quiver $C'$ obtained from $C$ by identifying vertices $x$ and $y$ when they correspond to modules which are merely graded-shifts of one another, where a ``graded-shift'' of a module $M$ means a module $M(i)$ defined by $M(i)_j=M_{i+j}$. By \cite[Theorem 3]{Auslander-Reiten:gradedCM}, $C'$ is naturally identified with a component of $\Gamma_{\Rhat}$. Therefore we might as well work over $R$, and just try not to keep track of the grading on $M$ and the grading on $\push(M)$ simultaneously. \end{remark}

\notation \label{notation:gamma and R'} 
Let $\gamma=y^{q-1}f/x \in Q(R)$. If $b \neq 0$, set $R'=S/(bx^p+y^q)S$; if $b=0$,  set $R'=S/yS\cong k[x]$.  In either case, $R'$ is a domain:

\begin{lem} \label{R' is a domain} If $b \neq 0$, then $S/(bx^p+y^q)S$ is a domain. \end{lem}
\begin{proof} As $S$ is factorial, it suffices to show $bx^p+y^q$ is irreducible.  Since a product $ss'$ fails to be homogeneous when either $s$ or $s'$ does, $bx^p+y^q$ is either irreducible or equal to a product of {homogeneous} nonunits. Let $s$ and $s'$ be homogeneous elements satisfying $ss'= bx^p+y^q$, and assume $s$ is a nonunit. Then $s$ has a term of the form $\alpha x^i$  for $\alpha \in k \setminus \{0\}$, so that $q| \deg s$. Likewise $p| \deg s$, and thus $\deg s=\deg (bx^p+y^q)$, hence $\deg s'=0$, so $s'$ is a unit.\qed \end{proof}

We will use the following piece of aritheoremetic several times. We omit the easy proof.
\begin{lem} \label{lem:numerical semigroup} 
If $b_1<q$ and $b_2<0$, or if  $b_1<0$ and $b_2<p$, then $  b_1 p+ b_2 q \notin p\NN +q \NN$.
\end{lem}

\begin{lem} \label{if rank neq 2} Let $(\varphi,\psi)$ be a reduced matrix factorization of $g$ and such that each indecomposable direct summand of $\cok \varphi$ has rank, and  char $k$ does not divide any of these ranks. Let $\alpha$ be a matrix such that  $\psi \alpha= \gamma \psi$ after passing to $R$. Then,
 $\push(\cok \varphi) = \cok \begin{pmatrix}\varphi & -\alpha\\0 & \psi\\ \end{pmatrix}$.\end{lem}

\begin{proof} By~\ref{gamma pushout}  we only need to check that $\gamma$  agrees with Notation~\ref{gamma_M}, and  the indecomposable summands of $\cok \varphi$ satisfy the hypotheses in Theorem~\ref{main theorem}.    If $b \neq 0$ then $R'=S/(bx^p+y^q)S$ and we take $\gamma'=y^{q-1}/x\in \Hom_{R'}(\m_{R'},R')$, and set $z=f$. Let $Q=Q(R)$ and $Q'=Q(R')$. As $\deg(y^{q-1}/x) =p(q-1)-q \notin p\NN +q \NN$ by Lemma~\ref{lem:numerical semigroup},  we have $\gamma' \in (R':_{Q'} \J(\Rpb)) \setminus R'$ by Lemma~\ref{a gamma crit}. So $\gamma=y^{q-1}f/x$ agrees with Notation~\ref{gamma_M}. If $b=0$, then $R'=S/yS$ and we take $\gamma' =1/x \in (R':_{Q'} \J(\Rpb)) \setminus R'$ and set $z=y^{q-1}f$. Again $\gamma=y^{q-1}f/x$ agrees with Notation~\ref{gamma_M}.
It only remains to note that $M \otimes _R Q'$ is a free $Q'$-module of rank equal to that of $M \otimes_R Q$, by Lemma~\ref{Q and Q'}. 
 \qed \end{proof}

 In preparation for what immediately follows, let us observe that $g-y^{q+v} \in x^m S$. Indeed, we have by assumption $f-y^v \in xS$, and $\deg f=\deg (y^v)=pv$. So if $x^iy^j$ is a monomial occurring in $f-y^v$, then we have $i > 0$, and $qi+pj=pv$. Since $\gcd(p,q)=1$, $i$ is therefore a positive multiple of $p$; in particular, $i > m$. Thus, if $\equiv$ denotes congruence modulo $x^m$, we have $f-y^v \equiv 0$, and $g-y^{q+v}=(bx^p+y^q)f-y^{q+v}\equiv y^q(f-y^v) \equiv 0$.

  Let 
 \begin{equation} \varphi=\begin{pmatrix} 
(g-y^{q+v})/x^m & -y^n\\
y^{q+v-n} & x^m\\ \end{pmatrix}, \text{ and }
\psi=\begin{pmatrix} 
x^m & y^n\\
-y^{q+v-n} & (g-y^{q+v})/x^m\\
\end{pmatrix}; \end{equation}
then $I \cong \cok \varphi$, and $(\varphi,\psi)$ is a matrix factorization of $g$. 
Let \begin{equation} \alpha=\begin{pmatrix}
0&-bx^{p-m-1}y^{n-1}f\\
x^{m-1}y^{q-n-1}f&0\\
\end{pmatrix},\end{equation} and note that $\psi \alpha=\gamma \psi$ after passing to $R$. Therefore if we let $\xi=\begin{pmatrix} \varphi& -\alpha\\ 0 & \psi\\ \end{pmatrix}$,  it follows  from  Lemma~\ref{if rank neq 2} that 
$ \cok \xi =\push(I)$. 

By Remark~\ref{MF morphism}, we can pick a matrix $\beta$, with entries in $S$, such that 
\begin{equation}\label {a phi=phi b} \alpha \varphi=\varphi \beta .\end{equation} In fact \begin{center}$\beta=\begin{pmatrix}
y^{q-1}(f-y^v)/x & -x^{m-1}y^{n-1}\\
y^{q-n-1}(bx^{p-m-1}y^vf+(f-y^v)(g-y^{q+v})/x^{m+1}) & -y^{q-1}(f-y^v)/x\\
\end{pmatrix}$. \end{center} We will never need to refer to the actual entries of $\beta$, though we will use that $\beta$ has no unit entries. By equation~\eqref{a phi=phi b},  the pair 
\begin{equation} \label{xi eta} (\xi, \eta) \text{ forms a matrix factorization of }g\text{, where }\xi=\begin{pmatrix} \varphi& -\alpha\\ 0 & \psi\\ \end{pmatrix}\text{, and }
\eta=\begin{pmatrix}
\psi & \beta\\
0 & \varphi\\
\end{pmatrix}.\end{equation} Furthermore, $(\xi,\eta)$ is a reduced matrix factorization.\\

To avoid extreme clutter, we will henceforth abuse notation!
\begin{caveat}
Regarding all matrices in this section, we from now on always take  the entries as living in $R$ rather than $S$, unless stated otherwise. \end{caveat}

The reader can check directly that $\alpha \varphi=-\gamma \varphi$.  In other words,
\begin{equation} \label{phi b=-g phi} \varphi \beta=-\gamma \varphi. \end{equation}

\defn \label{defn:W} We choose a matrix $W$ such that $\eta W= \gamma \eta$. Such $W$ exists by~\ref{gamma pushout}. Let $Z$ and $Z'$ be 2-by-2 matrices such that  $W=  \begin{pmatrix} \alpha &  Z'\\
0 & -\beta + \psi Z\\ \end{pmatrix}$. 

We explain why $W$ can be chosen to be of this form. To begin with, let $W$ be an  arbitrary matrix such that $\eta W= \gamma \eta$, and let $A$, $B$, $C$ and $D$ be 2-by-2 matrices such that $W=\begin{pmatrix} A & B\\
C& D\\ \end{pmatrix}$. The equation $ \begin{pmatrix}
\psi & \beta\\
0 & \varphi\\
\end{pmatrix} \begin{pmatrix} A & B\\
C& D\\ \end{pmatrix} = \begin{pmatrix}
\gamma \psi & \gamma \beta\\
0 & \gamma \varphi\\
\end{pmatrix}$ implies $\varphi D=\gamma \varphi$, which equals $-\varphi \beta$ (equation~\eqref{phi b=-g phi}). Therefore $\varphi (D+ \beta)=0$, and this implies $D+ \beta =\psi Z$ for some matrix $Z$. That we may choose $\begin{pmatrix} A \\ C\\ \end{pmatrix}$ to be $\begin{pmatrix} \alpha \\ 0\\ \end{pmatrix}$ follows from the equation $\psi \alpha =\gamma \psi$.

Now, let $\theta$ denote the 8-by-8 matrix  $\theta = \begin{pmatrix} \xi & -W\\
0 & \eta\\
\end{pmatrix}$. As $\rank (\cok \eta)=\rank (\cok \xi)=\rank(\push(I))=2$, Lemma~\ref{if rank neq 2} gives  $\cok \theta=\push(\cok \xi)=\push(\push(I))$. By Proposition~\ref{tube proposition}, in order to show the component of $\Gamma_{\Rhat}$ containing $I$ is a tube, it suffices to show that $\cok \theta=X \oplus Y \oplus F$, for some indecomposables $X$ and $Y$ and some possibly-zero free module $F$. It suffices to do this for $\im \theta$ instead of $\cok \theta$. We clarify that the term  the term ``indecomposable'' is unambiguous:

\begin{lem} \label{AR gdd lemma 1} \cite[Lemma 1]{Auslander-Reiten:gradedCM} Given an indecomposable $N$ in $L_p(R)$ (i.e., $N$ has no proper graded direct summand), we have that $\hat{N}$ is indecomposable in $L_p(\Rhat)$. In particular, $N$ is indecomposable as an $R$-module. \end{lem}

We state the above discussion as a lemma.

\begin{lem} \label{lem:to establish} In order to establish that the component of the AR quiver containing $\hat{I}$ is a tube, it suffices to show that $\im \theta=X \oplus Y$ for some graded modules $X$ and $Y$ each having no proper graded direct summand. \end{lem}

We begin by multiplying $\theta$ on the left and on the right by invertible matrices. 
Let $\id$ denote the 2-by-2 identity matrix, and let $H=\begin{pmatrix} \frac{1}{2} & 0\\ 0 & 1\\ \end{pmatrix}$.
Let $P'$ denote the 8-by-8 matrix

$P'=\begin{pmatrix} 
0 & \id & -\id & 0\\
\id & 0 &0&0\\
0&H&H&0\\
0&0&0&\id\\
\end{pmatrix}$, which is invertible with inverse $\begin{pmatrix} 
0 & \id & 0 & 0\\
\frac{1}{2}\id & 0 &\frac{1}{2}H^{-1}&0\\
-\frac{1}{2}\id&0&\frac{1}{2}H^{-1}&0\\
0&0&0&\id\\
\end{pmatrix};$ and let

 $P=\begin{pmatrix} 0 & \id&0&0\\
\frac{1}{2} \id &0&\id&0\\
-\frac{1}{2} \id &0&\id&-Z\\
0&0&0&\id\\ \end{pmatrix}$, which is invertible with inverse
$P^{-1}=\begin{pmatrix} 0 & \id & -\id & -Z\\
\id & 0 & 0 & 0\\
0 & \frac{1}{2}\id & \frac{1}{2}\id & \frac{1}{2}Z\\
0&0&0&\id \\ \end{pmatrix}$.

Now $P' \theta=
 \begin{pmatrix}
0 & \id & -\id & 0\\
\id & 0 &0&0\\
0&H&H&0\\
0&0&0&\id\\
\end{pmatrix}
\begin{pmatrix} \varphi & -\alpha & -\alpha & -Z'\\
0 & \psi & 0 & \beta-\psi Z\\
0 & 0 & \psi & \beta\\
0 & 0& 0& \varphi
\end{pmatrix}
=\begin{pmatrix} 0 & \psi & -\psi & -\psi Z\\
\varphi & -\alpha& -\alpha & -Z'\\
0 & H \psi & H \psi & H(2 \beta - \psi Z)\\
0&0&0&\varphi \\ 
\end{pmatrix}
$ , and

 $P' \theta P=\begin{pmatrix} 0 & \psi & -\psi & -\psi Z\\
\varphi & -\alpha& -\alpha & -Z'\\
0 & H \psi & H \psi & H(2 \beta - \psi Z)\\
0&0&0&\varphi \\ 
\end{pmatrix} 
\begin{pmatrix} 0 & \id &0&0\\
\frac{1}{2} \id &0&\id &0\\
-\frac{1}{2} \id &0&\id &-Z\\
0&0&0&\id \\ \end{pmatrix}
=\begin{pmatrix}\psi &0&0&0\\
0&\varphi &-2\alpha &\alpha Z-Z'\\
0&0&2H \psi & 2H(\beta-\psi Z)\\
0&0&0&\varphi \\ \end{pmatrix}$.

Let $c_j$ denote the $j$-th column of $P' \theta P$, $j=1,...,8$.
and let $M=\sum_{j=3}^8 Rc_j$. It remains to show that $M$ is an indecomposable module.
The module $M$ is graded if we take the following degrees for its generators. 

 \begin{center}
$\deg(c_3)=(v-n)p-mq$, $\deg(c_4)=-pq$, $\deg(c_5)=(v-n-1)p-q$,
 
 $\deg(c_6)=(v-1)p-(m+1)q$, $\deg(c_7)=(2v-n-2)p-(m+2)q+pq$, $\deg(c_8)=(v-2)p-2q$.\\
\end{center}

Assume $M=M' \oplus M''$ for some nonzero graded summands $M'$ and $M''$; now, by Lemma~\ref{lem:to establish}, producing a contradiction will complete our overall argument. Note that $\deg c_4< \deg c_j$ for all $j \ge 3$ different from 4. Since $R$ is connected, it  follows that $c_4$ lies in either $M'$ or $M''$; let us assume $c_4 \in M'$. Let $\pi \colon M \oplus (R c_1 + R c_2) \rightarrow M''$ denote the projection map onto $M''$, with the goal of showing that $\pi=0$. We have $\pi(c_1)=\pi(c_2)=\pi(c_4)=0$. Also immediate is $\pi(c_3)=0$ since $x^m c_3=y^{q+v-n } c_4$ and $x$ is a nonzerodivisor. 

Let $r_3,...,r_8 \in R$ be homogeneous elements such that $\sum_{j=3}^8 r_j c_j = \pi(c_5)$ and $\deg(r_j)=\deg(c_5)-\deg(c_j)$. Then each of $\deg(r_6)=-np+mq$, $\deg(r_7)=(-v+1)p+(m+1-p)q$, and $\deg(r_8)=(-n+1)p+q$ does not lie in  $\NN p + \NN q$ by Lemma~\ref{lem:numerical semigroup}, and so $r_6=r_7=r_8=0$. 

For a brief moment let us consider matrices with entries in $S$. Namely let $\tilde{W}$ denote a ``lift to $S$'' of the matrix $W$, and let $\tilde {\theta}$ be the lift of $\theta$, $\tilde \theta= \begin{pmatrix} \xi & -\tilde{W}\\
0 & \eta\\
\end{pmatrix}$. By the same reasoning used for the matrix factorization $(\xi,\eta)$, we know that $\tilde \theta$ is part of a matrix factorization $(\tilde \theta,\tilde \theta')$ where $\tilde \theta'=\begin{pmatrix} \eta & \tilde W' \\ 0 & \xi \end{pmatrix}$ for some 4-by-4 matrix $\tilde W'$. Let $\theta'=\tilde \theta' \otimes_S R$. 

As $\theta \theta'=0$, each column of matrix $P ^{-1} \theta'$ is a syzygy relation for the columns of $P' \theta P$. We can compute that the last four entries of  the column $P^{-1} \theta'_{\cdot, 4}$ are, in order, $-\frac{1}{2} y^n, \frac{1}{2}  x^m, 0, 0$. Therefore $\frac{1}{2} x^m c_6 \in \frac{1}{2} y^n c_5 + \sum_{j=1}^4 R c_j$. Then, $\pi(c_6)=\frac{y^n}{x^m}\pi(c_5)=\sum_{j=3}^5 (y^n/x^m) r_j c_j$. In particular $R$ must contain the fourth entry of this column:
 $\frac{y^n}{x^m}(r_3 y^{q+v-n}+r_4x^m-2r_5x^{m-1}y^{q-n-1}f) \in R$. Therefore, since $y^{q+v}/x^m \in R$, we have $2 r_5y^{q-1}f/x \in R$. Since $r_5 \in k$ and  char $k \neq 2$, this implies that either $r_5=0$ or $y^{q-1}f/x \in R$. If the latter were true,    
 then $rx=y^{q-1}f$ for some $r \in R$, and lifting $r$ to a preimage $s \in S$ we would have $sx -y^{q-1}f \in g S$. But  $sx -y^{q-1}f$ has nonzero $y^{q+v-1}$-term, whereas $\deg g=\deg y^{q+v}>\deg y^{q+v-1}$, so this is a contradiction. Hence $r_5=0$. Therefore $\pi(c_5) =r_3c_3+r_4c_4 \in \ker (\pi)$, hence $\pi(c_5)=0$ as $\pi$ is idempotent; and $\pi(c_6)=(y^n/x^m)\pi(c_5)=0$. 

Now we simply repeat the argument in order to show that $\pi(c_8)=\pi(c_7)=0$. For $r'_3,...,r'_8 \in R$  homogeneous such that $\sum_{j=3}^8 r'_j c_j = \pi(c_8)$ and $\deg(r'_j)=\deg(c_8)-\deg(c_j)$, each of $\deg(r'_5)=(n-1)p-q$, $\deg(r'_6)=-p+(m-1)q$, and $\deg(r'_7)=(-v+n)p+(m-p)q$ does not lie in  $\NN p + \NN q$ by Lemma~\ref{lem:numerical semigroup},  so $r'_5=r'_6=r'_7=0$. 
The last two entries of $P^{-1} \theta'_{\cdot, 7}$ are $x^m$ and $-y^{q+v-n}$, so we obtain  $x^m c_7 \in y^{q+v-n}c_8 + \sum_{j \le 6} Rc_j$, and therefore   $\pi(c_7)=(y^{q+v-n}/x^m)\pi(c_8)=(y^{q+v-n}/x^m)(r'_3c_3+r'_4c_4+r'_8c_8)$, whose fifth entry is 
$-r'_8 (y^{q+v-n}/x^m) W_{34}$. If $r'_8=0$ then $\pi(c_7) \in \ker \pi$ whence $0=\pi(c_7)=\pi(c_8)$; so, showing $r'_8=0$ is the last step. If $r'_8 \neq 0$ then it is a unit, and therefore $ (y^{q+v-n}/x^m) W_{34} \in R$. Then the lemma below would imply $y^{q+v-1}/x \in R$, and the reader can check that this is false.

\begin{lem} \label{lem:W_34} $W_{34}$, the (3,4)-th entry of the matrix $W$, lies in $k x^{m-1} y^{n-1} \setminus \{0\}$. \end{lem}
\begin{proof} Recall that $\eta W=\gamma \eta$ by definition of $W$. As $\eta_{4,4}=x^m$, we get $\gamma x^m=\eta_{4, \cdot} W_{\cdot,4}=y^{q+v-n} W_{34} +x^m W_{44}$. As $x$ is a nonzerodivisor and $\gamma \notin R$, we have $W_{34} \neq 0$. We naturally choose $W$ so that $\deg \eta _{ij} + \deg W_{jj'}=\deg (\gamma \eta_{ij'})$ for each $i, j, j'$. Setting $i=4$, $j=3$, $j'=4$, we have $\deg(W_{34})=\deg(\gamma \eta_{4,4})-\deg \eta_{4,3}=\deg(\gamma x^m)-\deg(y^{q+v-n})=\deg(y^{q-1} f x^{m-1})-\deg(y^{q+v-n})=(n-1)p+(m-1)q$. Since $p$ and $q$ are coprime, it follows that $W_{34} \in k x^{m-1} y^{n-1}$. 
\qed \end{proof}

\section{Another observation: $\syz([\gamma_M])$.} \label{sec:obs}

In this section, assume $R$ is a reduced complete local Gorenstein (but not regular) ring of dimension one, let $\m=\m_R$, $Q=Q(R)$,  and fix some indecomposable nonfree $M \in L_p(R)$. We aim to prove Proposition~\ref{syz gamma}, which (after some additional assumptions) states the relationship between the socle elements $[\gamma_M]$ and $[\gamma_{\syz M}]$ with respect to the $R$-algebra isomorphism $\syz_R \colon \stEnd_RM \to \stEnd_R(\syz_R (M))$. Let $M \Rbar$ denote the $\Rbar$-submodule of $M \otimes_R Q$ generated by $M$, and assume the following: $M\Rbar$ is a free $\Rbar$-module which possesses a basis consisting of elements in $M$. This is true if $R$ is a domain, since $\Rbar$ is in that case a DVR.

  \notation \label{e_i's} Fix $\gamma \in \J(\Rbar)^* \setminus R$, and fix  elements $e_1,...,e_n \in M$ forming a free $\Rbar$-basis for $M \Rbar$.   Given $h \in \End_R M $, let $\hbar$ denote the unique $\Rbar$-linear endomorphism of $M \Rbar$ extending $h$.  We regard $\hbar $ is an $n$-by-$n$ matrix with entries in $\Rbar$. Let  $I^{cd}=(R:_R \Rbar)$, the conductor ideal.

  \begin{lem} \label{gamma M Rbar}We have $\gamma M\Rbar \subseteq M$, and $I^{cd} (M \Rbar) \subset \bigoplus_i R e_i$. \end{lem}
  
 \begin{proof}  As $(R:_R \J(\Rbar))=(I^{cd}:_R \J(\Rbar))$, we have $\gamma \J(\Rbar) \subseteq I^{cd}$. Therefore $(\Rbar \gamma) \m \subseteq (\Rbar \gamma) \J(\Rbar) \subseteq I^{cd} \subseteq \m$, which says that $\Rbar \gamma \subseteq \End_R \m$. Since $M \cong \Hom_R(M,\m)^*$ is an $\End_R \m$-module,  we obtain $(\Rbar \gamma) M \subseteq M$, equivalently $\gamma M \Rbar \subseteq M$. That $I^{cd} (M \Rbar) \subset \bigoplus_i R e_i$ is obvious. \qed \end{proof}
 
 We have the following immediate consequence.
 \begin{lem} \label{matrix elements} Let $A \in \End_{\Rbar} (M \Rbar)$, i.e. $A$ is an $n \times n$ matrix with entries in $\Rbar$ (recall that we have a fixed basis, $\{e_1,...,e_n\}$). If each entry of $A$ lies in $\gamma \Rbar$, then $A$ sends $M$ into $M$. If each entry of $A$ lies in $I^{cd}$, then $A|_M: M \ra M$ is stably zero.
 \end{lem}

\begin{lem} \label{sigma with a single column}  There exists $f \in \End_RM$ satisfying the following conditions:

 (i) $[f]$ generates $\soc(\stEnd_RM)$;

(ii) all nonzero entries of $\fbar$ lie in $\Rbar \gamma$. 

(iii) the first column of $\fbar$ is its only nonzero column.

(iv) $\fbar_{1,1}=\gamma$. \end{lem}

\begin{proof} 
 If we take an $n \times n$ matrix $A$ with $A_{1,1}=\gamma$ and all other entries zero, then by Lemma~\ref{matrix elements}, $A=\hbar$ for some  endomorphism $h \in \End_RM$. As $\trace(h \otimes Q)=\trace(\hbar)=\gamma \notin R$,  $h$ is stably nonzero by Lemma~\ref{trace argument!}. Therefore by essentiality of the socle of $\stEnd _RM$, there exists  $g \in \End_RM$ such that $[gh]$ generates $\soc(\stEnd_RM)$. By Lemma~\ref{trace argument!}, there exists $h' \in \End_RM$ such that $\trace(ghh' \otimes_R Q) \notin R$, i.e. $\trace(h'gh \otimes_R Q) \notin R$. As $h'gh$ is stably nonzero by Lemma~\ref{trace argument!} once more, $[h'gh]$ generates $\soc(\stEnd_RM)$. Let $f=h'gh$. Now $\trace(f \otimes_R Q)=\trace(\fbar)=\fbar_{1,1} \in (\gamma \Rbar) \setminus R$. Therefore $\fbar_{1,1} =u \gamma $ for some unit $u \in \Rbar$. Finally, replacing $\fbar$ by $u^{-1} \fbar$, the result still sends $M$ into $M$, by Lemma~\ref{matrix elements}. \qed \end{proof}
 
 Note that $\syz_R$ is in general a well-defined functor on the stable category. In particular it gives an isomorphism of $R$-algebras $\stEnd_RM \to \stEnd_R(\syz_R (M))$. 
 
 For the remainder, assume $R$ is a domain, and assume $k=R/\m$ is algebraically closed.

 \begin{prop} \label{negative trace}  If $f \in \End_RM$ and $g\in \End_R(\syz_R M)$ are given such that $[f] \in \soc (\End_RM)$ and $[g] =\syz_R([f])$, then $\trace \fbar+\trace \gbar \in R$. \end{prop}
 
 \begin{proof}    By Lemma~\ref{trace argument!}, trace induces well-defined maps $\stEnd_RM \to \Rbar /R$ and $\stEnd_R(\syz_R M) \to \Rbar /R$. 
 As $\syz_R$ gives an isomorphism of $R$-algebras $\stEnd_R M \to \stEnd_R (\syz_R M)$, it restricts to an isomorphism  on socles, which are $R$-simple due to $k$ being algebraically closed. Because of these remarks, we can take our pick of $f$ and $g$, as long as $[f] \neq 0$ and $[g]=\syz_R([f])$; we will choose $f$ as in Proposition~\ref{sigma with a single column}. Let $n= \rank(M)$, and $\nu > n$ be the minimal number of generators of $M$. Let $\xi_1,..., \xi_\nu$ be a set of generators for $M$, such that $\{ e_1=\xi_1,...,e_n=\xi_n \}$ is an $\Rbar$-basis for $M \Rbar$.  For each $\xi_j$ we have an equation $\xi_j=\sum_{i=1}^n w_{i,j} e_i$, for $w_{i,j} \in \Rbar$. Since $\Rbar=R+\J(\Rbar)$ (due to $k$ being algebraically closed), we may assume  that for each $j>n$, and for each $i$, we have $w_{i,j} \in \J(\Rbar)$ and therefore  $w_{i,j} \gamma \in I^{cd}$.
 
 Take a free cover $\pi: F \to M$ sending $i$-th basis element to $\xi_i$. Since $f\in \End_RM$ is as in Proposition~\ref{sigma with a single column}, there is a $\nu \times \nu$ matrix $A:F \to F$ such that $\pi A=f \pi$, with the following properties. Columns 2 through $n$ of $A$ are zero. In addition, $A_{ij}=w_{1j}\fbar_{i1}$ for $(i,j) \in \{1,...,n\} \times \{n+1,...,\nu\}$, and $A_{ij}=0$ for  $(i,j) \in \{n+1,...,\nu\} \times \{n+1,...,\nu \}$. Set $N=\ker(\pi)$, and let $\vec{r}=[r_1,...,r_\nu]^T \in N$, that is, $\sum_{j=1}^\nu r_j \xi_j=0$. Recalling that $M \Rbar$ is free, and projecting onto the basis element $e_1$, we get $r_1+\sum_{j=n+1}^\nu r_j w_{1,j}=0$. If we set $\fbar_{i1}=0$ for $i >n$, then by definition of $A$ we have that the $i$-th entry of $A \vec{r}$ is $A_{i1}r_1+\sum_{j=n+1}^\nu w_{1j} \fbar_{i1} r_j=A_{i1}r_1+\fbar_{i1} \sum_{j=n+1}^\nu r_j w_{1,j}$, and by the above equation this equals $(A_{i1}-\fbar_{i1})r_1$. In other words, $A \vec{r}=r_1 \vec{v}$ where $\vec{v}=[v_1,...,v_{\nu}]^T \in F$ is given by $v_i =A_{i1}-\fbar_{i1}$. So if we let $g\in \End_RN$ be the restriction of $A$, we see that the image of $g$ has rank 1. We also see that the $A^2 \vec{r}=A (r_1 \vec{v})=r_1v_1 \vec{v}$, so that $v_1$, which equals $A_{1,1}-\gamma$, is an eigenvalue for $g$. Our goal is to show that $\trace(\gbar) +\gamma \in R$. Since $\trace(\gbar)=\trace(g \otimes_R Q)$ and $\im (g \otimes_R Q) \cong Q$, the following lemma finishes the proof. \qed \end{proof}
 
 \begin{lem} If $\phi:F \to F$ is an endomorphism of a free module over a domain $D$, with $\im(\phi)\cong D$, and $\lambda$ is an eigenvalue for $\phi$, then $\lambda=\trace(\phi)$. \end{lem}
 
 \begin{proof} Let $\vec{x}=[x_1,...,x_s]^T \in F$ generate the image of $\phi$. It is easily checked that $\phi(\vec{x})=\lambda \vec{x}$. Let $y_1,...,y_s \in D$ such that $\phi_{\cdot, j}=y_j \vec{x}$. Then $\lambda \vec{x}=\phi (\vec{x})=\sum_{j=1}^s x_j \phi_{\cdot, j}=\sum_{j=1}^s x_j y_j \vec{x}$. So $\lambda=\sum_{j=1}^s x_j y_j=\sum_j \phi_{j,j}=\trace(\phi)$. \qed \end{proof}

\begin{prop} \label{syz gamma}  We have 
 $\rank (\syz_R M) \cdot \syz([\gamma_M])+ \rank(M) \cdot [\gamma_{\syz M}]=0$.\end{prop}
 
 \begin{proof} Since $\soc \stEnd_R(M)$ is $R$-simple,  the map  $\soc \stEnd_R(M) \to \Rbar/R$, induced by trace, is injective. We also know that $[\gamma_M] \in  \soc \stEnd_R(M)$, by Theorem~\ref{main theorem}. Therefore $[\gamma_M]=\rank (M) \cdot [f]$ if we take $f \in \End_R M$ as in Lemma~\ref{sigma with a single column}; and Proposition~\ref{negative trace} implies $[\gamma_{\syz M}] =- \rank (\syz_R M ) \cdot \syz([f])$. The result follows. 
 \qed \end{proof}
 
 \section{Appendix.}\label{sec:ap}

In this appendix we record some lemmas for reduced connected graded rings of dimension one. 
The following theorem is well-known. 

\begin{theorem}\label{unramified} Let $B$ be one-dimensional, noetherian, local domain  with integral closure $\Bbar$ and $\m_B$-adic completion $\hat{B}$. Then the following are equivalent.

(1) $\hat{B}$ is a domain. (``$B$ is analytically irreducible.'') 

(2) $\Bbar$ is local and $\hat{B}$ is reduced. (``$B$ is unibranched and analytically unramified.'')

(3) $\Bbar$ is local and finitely generated as a $B$-module.
\end{theorem}

\notation \label{notation:Rhat} If $R$ is a connected graded ring, let $\Rhat$ denote the completion of $R$ with respect to its  graded maximal ideal, $\m$.

\begin{lem} \label{Rbar and Rhat are gdd} Let $R$ be a reduced connected graded ring. Then:

\begin{enumerate}
\item The integral closure of $R$ in $R[\text{nonzerodivisors}]^{-1}$ coincides with the integral closure of $R$ in $Q=R[\text{graded nonzerodivisors}]^{-1}$, our definition of $\Rbar$. Moreover, $\Rbar =\bigoplus_{i \ge 0} \Rbar_i$ is an $\NN$-graded subring of $Q$.
\item We have $\hat{R}= \prod_{i \ge 0} R_i$. 
\item The completion, $\Rhat$, is also reduced. If $R$ is a domain, then $\Rhat$ is a domain.
\item The integral closure, $\Rbar$, is finitely generated as an $R$-module.
\item The integral closure of the completion, $\overline{\Rhat}$, is finitely generated as an $\Rhat$-module.\end{enumerate}\end{lem}

\begin{proof} Statement (1) is \cite[Corollary 2.3.6]{Swanson-Huneke}. Statement (2) can be checked by noting that $\{\m^i\}_i$ is cofinal with $\{\bigoplus_{ j \ge i} R_j\}_i$, and checking that the completion of $R$ with respect to the latter filtration is isomorphic to $ \prod_{i \ge 0} R_i$. From (2) we see that $\hat{R}$ is reduced, resp. a domain, if $R$ is such. As $R$ is a finitely generated algebra over the field $R_0$, (4) is a consequence of \cite[Theorem 72]{Matsumura:1970}.  The last assertion is a consequence of Theorem~\ref{unramified} (alternatively, it follows from (4)). \qed \end{proof}

\begin{lem} \label{bar commutes with hat}  Let $D$ be a connected graded domain of dimension one, and let $\q=\bigoplus_{i \ge 1} D_i$, and  $\n=\bigoplus_{i \ge 1} \Dbar_i$. 
Then

(a) $\Dbar_0$ is a field,  and 

(b)  $\prod_{i \ge 0} \Dbar_i =\hat{\Dbar}^\n=\hat {\Dbar}^\q=\overline{\hat{D}} $.
\end{lem}

\begin{proof} 
 The  notation $\Dbar_i$ means $(\Dbar)_i$, and makes sense due to Lemma~\ref{Rbar and Rhat are gdd}, as does $\overline{\hat{D}}$. Since $\Dbar$ is an $\NN$-graded domain, $\n$ is  a prime ideal, and is thus maximal since $\dim \Dbar=\dim D=1$. So  $\Dbar_0$ is a field. Now $\prod_{i \ge 0} \Dbar_i=\hat{\Dbar}^\n$ by Lemma~\ref{Rbar and Rhat are gdd}. Note that $X_\n \neq 0$ for all graded $\Dbar$-modules $X \neq 0$. Now $\Dbar_\n/(\q \Dbar_\n)$ is an artinian  local ring, so there exists $i \ge 1$ such that  $ ((\n^i +\q\Dbar)/\q\Dbar )_\n =0$, hence $ (\n^i +\q \Dbar)/\q \Dbar =0$. Thus $\{\n^i\}_i$ and $\{\q^i \Dbar\}_i$ are cofinal, so  $\hat {\Dbar}^\n=\hat{\Dbar}^\q$. 
Lastly we show $\hat{\Dbar}^\q=\overline{\hat{D}}$.  Note  that $\Dbar \into \overline{\hat{D}}$, and since $\overline{\hat{D}}$ is complete by Lemma~\ref{Rbar and Rhat are gdd}, we have $ \overline{\hat{D}} \supseteq \hat{\Dbar}^q  \supseteq \Dhat$. It remains to observe that $\hat{\Dbar}^q $ is normal. But any $I$-adic completion of an excellent, normal ring, such as $\Dbar$, is normal (\cite[Theorem 79]{Matsumura:1970}). \qed \end{proof}

\begin{lem} \label {lem:polynomial ring}  Let $D$ be a connected graded domain of dimension one , and let $l= \min \{ i > 0 | \Dbar_i \neq 0 \}$.  Let $t$ be any nonzero element of $\Dbar_l$. Then $\Dbar=\bigoplus_{i\ge 0} \Dbar_0 t^i$ is the polynomial ring over the field $\Dbar_0$ in the variable $t$; and $\hat{\Dbar}=\prod_{ i\ge 0} \Dbar_i t^i$ is the power series ring. \qed \end{lem}

\begin{proof} By the previous lemmas, $\Dbar$ is  connected graded, so we may assume $D=\Dbar$ to improve notation. Then the previous lemma also shows that $\hat{D}=\prod_{i \ge 0} D_i$ is a  normal domain. Thus it is a DVR; let $\pi \in \hat{D}$ be a uniformizing parameter. So $\pi\hat{D}=\prod_{i \ge 1} D_i$.   Then $t=u \pi^i$ for some unit $u \in \prod_{i\ge 0} D_i$, and it follows that $i=1$,   hence $t$ is a uniformizing parameter for $\hat{D}$. It follows that $D_i=0$ for $i \notin \NN l$, and $D_i=D_0 t^{i/l}$ for $i \in \NN l$. The lemma follows. \qed \end{proof}

\begin{remark} \label{just a semigroup ring} If $D_0$ is algebraically closed, Lemmas~\ref{bar commutes with hat} and~\ref{lem:polynomial ring} show that $D$ has the form  $k[t^{i_1},...,t^{i_n}]$.\end{remark}

\bibliographystyle{spmpsci}      

\def\cprime{$'$} \def\polhk#1{\setbox0=\hbox{#1}{\ooalign{\hidewidth
  \lower1.5ex\hbox{`}\hidewidth\crcr\unhbox0}}}
  \def\polhk#1{\setbox0=\hbox{#1}{\ooalign{\hidewidth
  \lower1.5ex\hbox{`}\hidewidth\crcr\unhbox0}}}


\end{document}